\documentclass[11pt]{amsart}
\usepackage{macro}

\title{On free resolutions of Iwasawa modules}
\begin{document}
\author{Alexandra Nichifor}
\address{(Nichifor) Department of Mathematics, University of Washington, Seattle, USA 98195-4350}
\email{nichifor@u.washington.edu}

\author{Bharathwaj Palvannan}
\address{
(Palvannan) Department of Mathematics, University of Pennsylvania, Philadelphia, USA  19104-6395}
\email{pbharath@math.upenn.edu}


\begin{abstract}
Let $\Lambda$ (isomorphic to $\Z_p[[T]]$)  denote the usual Iwasawa algebra and $G$ denote the Galois group of a finite Galois extension $L/K$ of totally real fields.  When the non-primitive Iwasawa module over the cyclotomic $\Z_p$-extension has a free resolution of length one over the group ring $\Lambda[G]$, we prove that the validity of the non-commutative Iwasawa main conjecture allows us to find a representative for the non-primitive $p$-adic $L$-function (which is an element of a $K_1$-group) in a maximal $\Lambda$-order. This integrality result involves a study of the Dieudonn\'e determinant. Using a cohomolgoical criterion of Greenberg, we also deduce the precise conditions under which the non-primitive Iwasawa module has a free resolution of length one. As one application of the last result, we consider an elliptic curve over $\Q$ with a cyclic isogeny of degree $p^2$. We relate the characteristic ideal in the ring $\Lambda$ of the Pontryagin dual of its non-primitive Selmer group to two characteristic ideals, viewed as elements of group rings over $\Lambda$, associated to two non-primitive classical Iwasawa~modules.
\end{abstract}

\maketitle{}

%\tableofcontents

%\renewcommand{\baselinestretch}{1.1}

\section{Introduction}

Over the years, the Iwasawa main conjecture has been formulated in various setups and various guises. The underlying principle in each formulation has been to relate objects on the algebraic side to the objects on the analytic side. On the algebraic side of Iwasawa theory, one studies modules over Iwasawa algebras. An Iwasawa algebra is a completed group ring $\Z_p[[\G]]$, for some $p$-adic Lie group $\G$.  On the analytic side, one studies $p$-adic $L$-functions. The $p$-adic $L$-functions are believed to satisfy certain integrality properties. For example, consider the case when the group $\G$ is isomorphic to $\Z_p \times \Delta$, for some finite abelian group $\Delta$. Under suitable conditions, the $p$-adic $L$-function is known to be a measure (not just a pseudo-measure). Our results in this paper are motivated by similar integrality properties of $p$-adic $L$-functions, in the non-commutative setting, as predicted by the non-commutative Iwasawa main conjectures.

Throughout this paper, fix $p$ to be an odd prime. Let us first introduce all the notations that will be required to describe our results precisely. Let $L/K$ be a finite Galois extension of totally real fields. Let $\chi : \Gal{\overline{\Q}}{K} \rightarrow  \mathbb{F}_p^\times \hookrightarrow \Z_p^\times$ be a finite character that is either totally even or totally odd. For the sake of simplicity, we have chosen to work with a finite character $\chi$ taking values in $\Z_p^\times$. One could also consider a finite character taking values in an unramified extension of $\Q_p$. Our results would hold analogously.

We let $K_\chi$ denote the  number field $\overline{\Q}^{\ker(\chi)}$. We let $L_\chi$ denote the compositum of $L$ and $K_\chi$. We let the fields $K_\infty$, $K_{\chi,\infty}$, $L_\infty$ and $L_{\chi,\infty}$ denote the cyclotomic $\Z_p$-extensions of $K$, $K_\chi$, $L$ and $L_\chi$ respectievly. Let $G:=\Gal{L}{K}$, $\Delta := \Gal{K_\chi}{K}$ and $\Gamma :=\Gal{K_\infty}{K}$. Throughout this paper, we will impose the following condition:
\begin{align}
& \label{field-condn} K_\infty \cap L_\chi = K.
\end{align}
Condition (\ref{field-condn}) imposed above allow us to view $\chi$ naturally  as a character of the groups  $\Gal{K_\chi}{K}$, $\Gal{L_\chi}{L}$, $\Gal{K_{\chi,\infty}}{K_\infty}$ and $\Gal{L_{\chi,\infty}}{L_{\infty}}$ (and throughout this paper, we shall take this point of view). We have the following field diagrams and natural isomorphisms in mind:

{\footnotesize 
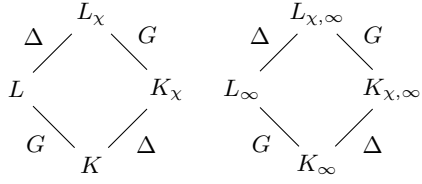
\begin{center}
\begin{tikzpicture}[node distance = 1cm, auto]
      \node (K) {$K$};
      \node (L) [above of=K, left of=K] {$L$};
      \node (Kchi) [above of=K, right of=K] {$K_\chi$};
      \node (Lchi) [above of=K, node distance = 2cm] {$L_\chi$};
      \draw[-] (K) to node {$G$} (L);
      \draw[-] (K) to node [swap] {$\Delta$} (Kchi);
      \draw[-] (L) to node {$\Delta$} (Lchi);
      \draw[-] (Kchi) to node [swap] {$G$} (Lchi);

\node (Kinf)  [right of = K, node distance = 3cm] {$K_\infty$};
      \node (Linf) [above of=Kinf, left of=Kinf] {$L_\infty$};
      \node (Kchiinf) [above of=Kinf, right of=Kinf] {$K_{\chi,\infty}$};
      \node (Lchiinf) [above of=Kinf, node distance = 2cm] {$L_{\chi,\infty}$};

      \draw[-] (Kinf) to node {$G$} (Linf);
      \draw[-] (Kinf) to node [swap] {$\Delta$} (Kchiinf);
      \draw[-] (Linf) to node {$\Delta$} (Lchiinf);
      \draw[-] (Kchiinf) to node [swap] {$G$} (Lchiinf);
     \end{tikzpicture}
  \end{center}}
We will consider the non-primitive classical Iwasawa module $\X$ (defined in the next paragraph). Let $\Lambda$ denote the completed group ring $\Z_p[[\Gamma]]$. Let $\Lambda[G]$ denote the group ring over $\Lambda$.  In what follows, we will assume that all the $\Lambda[G]$-modules are left $\Lambda[G]$-modules with a left $G$-action. It turns out that $\X$ is a finitely generated torsion module over $\Lambda[G]$. See \cite[Proposition 1]{greenberg2014p}.   We refer the reader to this work of Greenberg to see how the Iwasawa module $\X$ relates to Galois groups appearing in classical Iwasawa theory.

Let $\mathfrak{D}(\chi)$ equal $\frac{\Q_p(\chi)}{\Z_p(\chi)}$, that is,  $\frac{\Q_p}{\Z_p}$ with an action of $\Gal{\overline{\Q}}{K}$ via the character $\chi$.  Let us define the (non-primitive) Selmer group $\Sel_{\mathfrak{D}(\chi)}(L_\infty)$.  The definition of the Selmer group depends on the parity of $\chi$. Let $\Sigma$ denote a finite set of primes in $K$ containing the primes above $p$, $\infty$, a finite prime number $\nu_0$ not lying above $p$, and all the primes ramified in the extensions $L/K$ and $K_\chi/K$. For any algebraic extension $F$ of $K$, we let $\Sigma_p(F)$ denote the set of all primes above $p$ in $F$. We let $\Sigma_0$ equal the set $\Sigma \setminus \Sigma_p(K)$. We let $K_\Sigma$ denote the maximal extension of $K$ that is unramified outside~$\Sigma$.  \\

If the character $\chi$ is totally even, the Selmer group is defined as follows:
 \begin{align*}
\Sel_{\mathfrak{D}(\chi)}(L_\infty) =  H^1\left(\Gal{K_\Sigma}{L_\infty}, \mathfrak{D}(\chi)\right). \end{align*}
If the character $\chi$ is totally odd, the Selmer group is defined as follows:
 \begin{align*}
\Sel_{\mathfrak{D}(\chi)}(L_\infty) = \ker \bigg( H^1\left(\Gal{K_\Sigma}{L_\infty}, \mathfrak{D}(\chi)\right) \xrightarrow {\phi^{\Sigma_0}_{\mathfrak{D}(\chi),\text{odd}}} \prod_{\omega \in \Sigma_p({L_\infty}),  } H^1\left(I_\omega,\mathfrak{D}(\chi)\right)^{\Gamma_\omega} \bigg).
\end{align*}
Here, $I_\omega$ denotes the inertia subgroup inside the decomposition group $G_\omega$ corresponding to the prime $\omega$. We let $\Gamma_\omega$ denote the quotient group $\frac{G_\omega}{I_\omega}$. The map $\phi^{\Sigma_0}_{\mathfrak{D}(\chi),\text{odd}}$ denotes the (natural) restriction map. Let $\X$ denote the Pontryagin dual of $\Sel_{\mathfrak{D}(\chi)}(L_\infty)$. We will sometimes write $\X_{\mathfrak{D}(\chi)}(L_\infty)$ when we want to emphasize the field $L$ and the character $\chi$.

\begin{remark}
We work with non-primitive Selmer groups since one can use Greenberg's results to show that the global-to-local map $\phi^{\Sigma_0}_{\mathfrak{D}(\chi),\text{odd}}$ is surjective. We include the auxillary prime $\nu_0$ not lying above $p$ to ensure that we are working with Selmer groups that are genuinely non-primitive. When $\chi$ equals the Teichm\"{u}ller character $\omega$, the global-to-local map defining the primitive Selmer group is not surjective. See \cite[Proposition 5.3.3]{greenberg2010surjectivity} and the illustration that follows.
\end{remark}

\subsection{Integrality property for the non-primitive $p$-adic $L$-function} \mbox{}

The main conjecture (Conjecture \ref{conj:noncomm}) allows us to deduce certain integrality properties for the non-primitive $p$-adic $L$-function $\xi$ from the non-primitive Iwasawa module $\X$. Let $Q_\Lambda$ denote the fraction field of $\Lambda$. On the algebraic side, one considers an element in the relative $K_0$-group $K_0\left(\Lambda[G],Q_\Lambda[G]\right)$. On the analytic side, we have a non-primitive $p$-adic $L$-function $\xi$ in $K_1\left(Q_\Lambda[G]\right)$. The interpolation properties of the $p$-adic $L$-function $\xi$ are recalled in Section \ref{subsec:inter_prop}. Under the connecting  homomorphism $\partial: K_1\left(Q_\Lambda[G]\right)\rightarrow K_0\left(\Lambda[G],Q_\Lambda[G]\right)$ in $K$-theory, the non-commutative Iwasawa main conjecture relates the non-primitive $p$-adic $L$-function $\xi$ to the element on the algebraic side in the relative $K_0$-group. Works of Ritter-Weiss (\cite{MR2813337}) and Kakde (\cite{MR3091976}) independently show that the non-commutative Iwasawa main conjecture holds when $\chi$ is totally even, assuming Iwasawa's $\mu=0$ conjecture holds. Progress towards the Iwasawa main conjecture,  without assuming the validity of Iwasawa's $\mu=0$ conjecture, has been made in recent work of Johnston-Nickel \cite{johnstonhybrid}.

The Artin-Wedderburn theorem gives us the following isomorphism:
\begin{align} \label{artin-wedderburn}
Q_\Lambda[G] \cong \prod_{i} M_{m_i}(D_{\Lambda,i}),
\end{align}
where  the $D_{\Lambda,i}$'s are  division algebras, finite-dimensional over $Q_\Lambda$. One obtains the following sequence of isomorphisms:
\begin{align} \label{K1-abelian}
K_1\left(Q_\Lambda[G]\right) \cong \prod_{i} K_1\left(M_{m_i}(D_{\Lambda,i})\right)  \overset{\det}{\cong} \prod_{i} \left(M_{m_i}\left(D_{\Lambda,i}\right)^*\right)^{ab} \cong \left(Q_\Lambda[G]^*\right)^{\ab}.
\end{align}

Here, $\left(Q_\Lambda[G]^*\right)^{\ab}$ denotes the maximal abelian quotient of the multiplicative group of units in the ring $ Q_\Lambda[G]$. The $\det$ in equation (\ref{K1-abelian}) refers to the Dieudonn\'e determinant. Its definition is recalled in Section \ref{determinant}. One has a natural surjection $ Q_\Lambda[G]^* \twoheadrightarrow K_1\left( Q_\Lambda[G]^*\right)$ of groups. One can ask the following question:

\begin{question} \label{INT} When $\chi$ is non-trivial, does $\xi$ belong to the image of the following natural map of multiplicative monoids?
$$\Lambda[G] \cap Q_\Lambda[G]^* \rightarrow K_1\left(Q_\Lambda[G]\right).$$
\end{question}

A similar question, pertaining to the integrality properties of $p$-adic $L$-functions, was raised in the five author paper \cite{MR2217048}. See Conjecture 4.8 in \cite{MR2217048}. In that paper, the authors considered $p$-adic $L$-functions associated to ordinary elliptic curves. Note that the formulation of Question \ref{INT} is stronger than Conjecture 4.8 in \cite{MR2217048} as the authors of \cite{MR2217048} state their conjecture assuming that the group $G$ has no element of order $p$. One can also consider this question as a (non-commutative, non-primitive) refinement of the $p$-adic Artin conjecture of Greenberg \cite{MR692344}.

\begin{remark}
In the setup of our theorems, the $p$-adic $L$-function (if it exists, as is conjectured) turns out to be unique. To see this, it suffices to show that the reduced Whitehead group $\mathrm{SK}_1\left(Q_\Lambda[G]\right)$ equals zero. This, in turn, reduces to showing that $\mathrm{SK}_1\left(D_{\Lambda,i}\right)$ equals zero for each of the division algebras $D_{\Lambda,i}$ appearing in equation (\ref{artin-wedderburn}). In Section \ref{sec:maxorder}, we show that each of these division algebras $D_{\Lambda,i}$ are of the form $D \otimes_{\Q_p} Q_\Lambda$, where $D$ is a finite dimensional division algebra over $\Q_p$. A result of Nakayama-Matsushima \cite{MR0014081} shows that for finite dimensional division algebras $D$ over $\Q_p$, we have $\mathrm{SK}_1(D)=0$. This result of Nakayama-Matsushima combined with the fact that $Q_\Lambda$ is a purely transcendental extension of $\Q_p$ along with Platanov's Stability Theorem (\cite[Page 315]{MR562621}) is then sufficient to show that $\mathrm{SK}_1\left(D_{\Lambda,i}\right)$ equals zero for each the division algebras $D_{\Lambda,i}$ appearing in equation~(\ref{artin-wedderburn}).
\end{remark}

Let $M_{\Lambda[G]}$ denote the maximal $\Lambda$-order inside $Q_\Lambda[G]$ containing $\Lambda[G]$ as defined in equation (\ref{maximal-order-definition}). Inside $Q_\Lambda[G]$, we have the inclusions $$\Lambda[G] \subset M_{\Lambda[G]} \subset \frac{1}{|G|}\Lambda[G].$$
We prove the following partial result towards Question \ref{INT}.

\begin{restatable}{Theorem}{intdeterminant}\label{integrality-determinant}
Suppose $\chi$ is non-trivial and Conjecture \ref{conj:noncomm} holds. Suppose also that the $\Lambda[G]$-module $\X$ has a free resolution of length one. Then, the non-primitive $p$-adic $L$-function $\xi$ belongs to the image of the following natural map of multiplicative monoids:
\begin{align*}
M_{\Lambda[G]} \cap Q_\Lambda[G]^* \rightarrow K_1\left(Q_\Lambda[G]\right).
\end{align*}
\end{restatable}

Suppose $\chi$ is non-trivial and Conjecture \ref{conj:noncomm} holds. When $p$ does not divide the order of $G$, the $\Lambda[G]$-module $\X$ has a free resolution of length one (see Remark \ref{nmidG}) and the maximal $\Lambda$-order $M_{\Lambda[G]}$ coincides with $\Lambda[G]$.  One has an affirmative answer to Question \ref{INT}. When $p$ divides the order of $G$, the maximal $\Lambda$-order $M_{\Lambda[G]}$ containing $\Lambda[G]$ does not coincide with $\Lambda[G]$. When $G$ is abelian, Question \ref{INT} has an affirmative answer. In the commutative case, the question amounts to asking whether the non-primitive $p$-adic $L$-function is a measure (and not just a pseudo-measure). It is already known that the non-primitive (abelian) $p$-adic $L$-function is a measure, due to works of Barsky \cite{MR525346}, Cassou-Nogu\`es \cite{MR522119} and Deligne-Ribet \cite{MR579702}. These results in the commutative case served as an additional source of motivation for us to pursue this question in the non-commutative setting.

\begin{comment} One crucial input to Theorem \ref{integrality-determinant} is the fact that the maximal order $M_{\Lambda[G]}$ has finite global dimension. This allows us to formulate Conjecture \ref{conj:noncomm} using the $M_{\Lambda[G]}$-module $M_{\Lambda[G]} \otimes_{\Lambda[G]} \X$. When the global dimension of a maximal order over a regular local dimension of dimension two is finite, Ramras (\cite[Theorem 1.10]{MR0245572}) has established an analog of the Auslander-Buchsbaum formula. Results of Greenberg \cite{greenberg2014pseudonull} establish that $\X$ has no non-trivial pseudo-null $\Lambda$-submodules. Combining these results would let us deduce that the $M_{\Lambda[G]}$-module $M_{\Lambda[G]} \otimes_{\Lambda[G]} \X$ has a free resolution of length one.
\end{comment}

\begin{remark}
A variant of Conjecture \ref{conj:noncomm} involving the maximal order $M_{Z(Q_\Lambda[G])}$ of the center $Z(Q_\Lambda[G])$ of $Q_{\Lambda[G]}$ is known. See work of Johnston-Nickel \cite[Theorem 4.9]{johnstonhybrid} or Ritter-Weiss \cite[Theorem 16 and Remark (H)]{MR2114937} for the exact statement. In light of those results, it may be helpful to remark that the conclusion of Theorem \ref{integrality-determinant} would follow without requiring the validity of Conjecture \ref{conj:noncomm} if the reduced norm $\mathrm{Nrd}: M_{\Lambda[G]}^\times \rightarrow M_{Z(Q_\Lambda[G])}^\times$ is surjective. Since $p$ is odd, the extension $Q_{\Lambda}[G] \otimes_{\Q_p}\Q_p(\mu_{p})$ is a product of matrix rings over commutative fields. See \cite[Theorem 1.10(ii)]{MR933091}. One way to bypass requiring the validity of Conjecture \ref{conj:noncomm} in Theorem \ref{integrality-determinant} then is to simply work with the pair $\left(\Lambda[G] \otimes_{\Z_p}\Z_p[\mu_{p}], \ Q_{\Lambda}[G] \otimes_{\Q_p}\Q_p(\mu_{p})\right)$ instead of the pair $(\Lambda[G],Q_{\Lambda}[G])$.
\end{remark}

\subsection{Free resolutions of length one over $\Lambda[G]$} \label{subsec:free_res} \mbox{}

One can ask when the $\Lambda[G]$-module $\X$ has a free resolution of length one. Over the integral group ring $\Lambda[G]$, the situation is much easier to handle when the order of $G$ is co-prime to $p$. See Remark \ref{nmidG}. When $p$ does not divide $|G|$, the global dimension of the ring $\Lambda[G]$ equals two; in this case the $\Lambda[G]$-module $\X$ has a free resolution of length one. The situation is more complicated when there exists an element of order $p$ in $G$ since the ring $\Lambda[G]$ would then have infinite global dimension; in this case it is possible for $\Lambda[G]$-modules to have no non-trivial pseudo-null submodules and yet have infinite projective dimension over $\Lambda[G]$. The purpose of Theorems \ref{evenprojectivity} and \ref{oddprojectivity} is to precisely circumvent these difficulties.

\begin{restatable}{Theorem}{evprojectivity}\label{evenprojectivity}
Suppose $\chi$ is totally even. Suppose $p$ divides $|G|$. The $\Lambda[G]$-module $\X$ has a free resolution of length one if and only if $\chi$ is non-trivial.
\end{restatable}

\begin{restatable}{Theorem}{odprojectivity} \label{oddprojectivity}
Suppose $\chi$ is totally odd. Suppose $p$ divides $|G|$. The $\Lambda[G]$-module $\X$ has a free resolution of length one if and only if one of the two following conditions holds for every prime $\omega \in \Sigma_p(L_\infty)$:
\begin{enumerate}[label=(\Roman*), ref=\Roman*]
\item\label{non-anomalous} $H^0\left(G_\omega,\mathfrak{D}(\chi)\right)=0$
\item \label{tamelyramified} $\omega$ is tamely ramified in the extension $L_\infty / K_\infty$.
\end{enumerate}
\end{restatable}

In the first author's thesis \cite{nichifor2004iwasawa} in 2004,  Theorems \ref{evenprojectivity} and  \ref{oddprojectivity} were proved in the case when $G$ is a cyclic $p$-group using a formula of Kida (\cite{MR599821}) and assuming the validity of Iwasawa's $\mu=0$ conjecture. The results in this paper are a natural generalization of the results of \cite{nichifor2004iwasawa}, though the methods in this paper are significantly different. We use a cohomological criterion developed by Greenberg in \cite{greenberg2011iwasawa}. This allows us to prove our results, without having to assume the validity of Iwasawa's $\mu=0$ conjecture. See Proposition 3.1.1 in \cite{greenberg2011iwasawa} for similar results concerning the Pontryagin duals of Selmer groups associated to elliptic curves. Much of our motivation towards this paper stems from this work of Greenberg.

Nickel has shown that when the character $\chi$ is odd and when all the primes $\omega$ in $\Sigma_p(L_\infty)$ are \textit{almost tame}, then the $\Lambda[G]$-module $\X$ has a free resolution of length one. The condition that a prime $\omega$ in $\Sigma_p(L_\infty)$ is \textit{almost tame} is related to Condition \ref{tamelyramified} in Theorem \ref{oddprojectivity} and the image of complex conjugation in the decomposition group corresponding to $\omega$. See Proposition 4.1 in \cite{MR2822866} and Proposition 7 in \cite{MR2805422}. Theorem \ref{evenprojectivity} can also be deduced from the machinery of Selmer complexes appearing in the work of Fukaya and Kato \cite{MR2276851}, as we indicate in Section \ref{section-non-commutative}. For Theorem \ref{evenprojectivity}, one can also use results from the work of Ritter and Weiss \cite{MR1935024,MR2114937}. See also Chapter 5 of Witte's habilitation thesis \cite{witte}.

\begin{remark}
In light of results of Nickel \cite{MR2805422,MR2822866} and Ritter-Weiss \cite{MR1935024,MR2114937}, Theorems \ref{evenprojectivity} and \ref{oddprojectivity} are not essentially new. However, one of our main objectives in this paper is to initiate an approach towards studying integrality properties of $p$-adic $L$-functions (as in Question \ref{INT}) in general situations using Greenberg's cohomological criterion (\cite[Proposition 2.4.1]{greenberg2011iwasawa})  and the Dieudonn\'e determinant. Greenberg's cohomological criterion is valid over general $1$-dimensional $p$-adic Lie groups. Our approach via the theory of Dieudonn\'e determinant may be of independent interest from the perspective of non-commutative algebra. In the non-commutative setting, the study of the Dieudonn\'e determinant seems a rather subtle question to us. See Examples \ref{eg1}, \ref{eg2}, \ref{dihedral-example} and \ref{quaternion-example} in Section \ref{determinant}. These examples suggest that establishing that the non-primitive Iwasawa module has a free resolution of length $1$ alone may not be sufficient to establish an affirmative answer towards Question~\ref{INT} (whenever it does have an affirmative answer).
\end{remark}

\subsection{Elliptic curves with a cyclic $p^2$ isogeny}

We will now consider applications of Theorems \ref{evenprojectivity} and \ref{oddprojectivity} to a setting involving an elliptic curve defined over $\Q$ with a cyclic $p^2$ isogeny. Our main theorem (Theorem \ref{elliptic-curves}) in this setting  is a generalization of a result that appears implicitly in the work of Greenberg and Vatsal \cite{greenberg2000iwasawa}. For a generalization of this work of Greenberg and Vatsal in another direction, see work of Hirano \cite{MR3525165}.

Let $E$ be an elliptic curve defined over $\Q$ with good ordinary or split multiplicative reduction at $p$. Let $\Phi : E \rightarrow E'$ be a cyclic isogeny over $\Q$ of degree $p^2$. Let $\tilde{\Phi} : E' \rightarrow E$ be the dual isogeny. We shall suppose that the Galois action on the kernel of the isogeny $\Phi$ is even.  We will first state the theorem in this setting before explaining the notations.

\begin{restatable}{Theorem}{ellipticcurves} \label{elliptic-curves}
Suppose the even character $\chi_\phi$ is ramified at $p$. Suppose the condition \ref{Non-DG} holds. We have the following equality of ideals in $\frac{\Lambda}{(p^2)}$:
\begin{align} \label{equality-elliptic-curves}
\sigma_{p^2}\bigg(\mathrm{Char}_{\Lambda}\left(\Sel_{E[p^\infty]}(\Q_\infty)^\vee\right)\bigg) = \sigma_{\phi}\bigg(\left(\det\left(A_\phi\right)\right)\bigg)  \sigma_{\tilde{\phi}} \bigg(\left(\det\left(A_{\tilde{\phi}}\right)\right)\bigg).
\end{align}
\end{restatable}

The non-primitive Selmer group $\Sel_{E[p^\infty]}(\Q_\infty)$ associated to $E$ is defined in the work of Greenberg and Vatsal \cite{greenberg2000iwasawa}. The characteristic ideal of the $\Lambda$-module $\Sel_{E[p^\infty]}(\Q_\infty)^\vee$ is denoted by $\mathrm{Char}_{\Lambda}\left(\Sel_{E[p^\infty]}(\Q_\infty)^\vee\right)$. We will need to consider the natural map $\sigma_{p^2}:\Lambda \rightarrow \frac{\Lambda}{(p^2)}$. Here, $G_{\phi}$ and $G_{\tilde{\phi}}$ are abelian Galois groups, of order dividing $p$, of Galois extensions $L_{\phi}/\Q$ and $L_{\tilde{\phi}}/\Q$ respectively. These fields are ``cut out'', in a certain sense, by the cyclic isogenies $\Phi$ and $\tilde{\Phi}$. We will need to impose the condition that $L_\phi \cap \Q_\infty =\Q$, similar to the condition given in (\ref{field-condn}). This condition is labeled \ref{Non-DG}. See Section \ref{ellipticcurves-section} for the precise definitions of the various objects along with the description of the  ring homomorphisms $\sigma_{\phi} : \Lambda[G_\phi] \rightarrow \frac{\Lambda}{(p^2)}$ and  $\sigma_{\tilde{\phi}} : \Lambda[G_{\tilde{\phi}}] \rightarrow \frac{\Lambda}{(p^2)}$ given in (\ref{groupring-homo}). The Galois action on $\ker(\Phi)[p]$ is given the character $\chi_\phi$. In this setup, as we shall see in Section \ref{ellipticcurves-section}, Theorems \ref{evenprojectivity} and \ref{oddprojectivity} will allow us to consider two non-primitive Iwasawa modules that have free resolutions of length one over $\Lambda[G_\phi] $ and $\Lambda[G_{\tilde{\phi}}]$ respectively. These free resolutions of length one will naturally lead us to consider two square matrices $A_\phi$ and $A_{\tilde{\phi}}$ in group rings $\Lambda[G_\phi]$ and $\Lambda[G_{\tilde{\phi}}]$  respectively. See equation (\ref{2-classical-ses}).

\section{Dieudonn\'e determinant} \label{determinant}

To answer Question \ref{INT}, that deals with the integrality properties of non-primitive $p$-adic $L$-functions, we will need to develop some preliminaries on the Dieudonn\'e determinant. We shall follow some of the terminology introduced in Lam's book on non-commutative rings \cite{MR1838439}. In our discussions, the rings will always be associative rings with a unity. The units of a ring $T$, denoted by $T^*$, will consist of elements that have both a left and a right inverse. We will say that a ring $T$ is Noetherian if it is both left and right Noetherian. We will say that a Noetherian ring $T$ is a local ring if it has a unique maximal left ideal $\m_T$ (this ideal $\m_T$ turns out to be the unique  maximal right ideal). \\

We will say that a Noetherian ring $T$ is a semi-local ring if the quotient ring $\frac{T}{\text{Jac}(T)}$ is semisimple. Here, $\text{Jac}(T)$ is the Jacobson radical of $T$ (which is a two-sided ideal in $T$). Let $T$ be a semi-local ring. We recall three properties associated to it:
\begin{enumerate}
\item The matrix ring $M_n(T)$ is a semi-local ring with Jacobson radical equal to $M_n\left(\mathrm{Jac}(T)\right)$. See 20.4 in Lam's book \cite{MR1838439}.
\item A semi-local ring is Dedekind-finite. That is, whenever an element $u$ is right-invertible, then $u$ is left-invertible (or equivalently, whenever an element $u$ is left-invertible, then $u$ is right-invertible). See Proposition 20.8 in Lams' book \cite{MR1838439}.
\item A matrix $A$ in $M_n(T)$ is invertible if and only if it becomes invertible in $M_n\left(\frac{T}{\mathrm{Jac}(T)}\right)$.  See Theorem 1.11 in Oliver's book \cite{MR933091}. In particular, an element $u$ in $T$ is invertible if and only if its image in the quotient ring $\frac{T}{\mathrm{Jac}(T)}$ is invertible.
\end{enumerate}

For each integer $n \geq 1$, one can consider the inclusions $\Gl_n(T) \hookrightarrow \Gl_{n+1}(T)$ via the map $g \rightarrow \left(\begin{array}{cc} g & 0 \\ 0 & 1 \end{array} \right)$. Let $\Gl_\infty\left(T\right)$ equal $\bigcup \limits_{n \geq 0} \Gl_n\left(T\right)$.  The group $K_1(T)$ is defined below:

\begin{align*}
K_1\left(T\right) := \frac{\Gl_\infty(T)}{[\Gl_\infty(T),\Gl_\infty(T)]}.
\end{align*}

Here,  $\left[\Gl_\infty(T), \Gl_\infty(T)\right]$ is the commutator subgroup of $\Gl_\infty(T)$. We will need to consider a subgroup $W(T)$, of the multiplicative group $T^*$, generated by elements of the form $(1+rs)(1+sr)^{-1}$, whenever $1+rs$ is a unit in the ring $T$. Here, $r,s$ are elements of the ring $T$. The group $W(T)$ contains the commutator subgroup $[T^*,T^*]$.  One has a natural ``determinant'' map, often called the Dieudonn\'e or Whitehead determinant. See Example 1.3.7 and Exercise 1.2, both in Chapter III of Weibel's $K$-book \cite{MR3076731} for more details. The Dieudonn\'e determinant is the unique group homomorphism:
\begin{align*}
\det: K_1(T) \rightarrow \frac{T^*}{W(T)},
\end{align*}
characterized by the following properties:

\begin{enumerate}
\item If $A$ is an elementary $n \times n$ matrix, then $\det(A)=1$.  We say that an $n \times n$ matrix $A = (a_{ij})$ is elementary if there exists distinct indices $r,s$ ($r \neq s$) and an element $\lambda$ in $T$ such that $a_{ij} = \left\{ \begin{array}{cc} 1, \quad & \text{if } i = j,  \\ \lambda, \quad & \text{ if } i=r, \ j=s, \\ 0 \quad & \text{otherwise}. \end{array}  \right.$
\item If $\mathrm{diag}(t) =$ $\left(\begin{array}{cccc} t & 0 & \cdots & 0 \\ 0 & 1 & \cdots & 0 \\ \vdots & \vdots & \ddots & \vdots \\ 0 & 0  & \cdots & 1 \end{array}\right)$ and $g$ belong to $\Gl_n(T)$, then $\det\big(\mathrm{diag}(t)g\big) = t \det(g)$.
\end{enumerate}

In fact, Vaserstein \cite{MR2111217} has shown that the Dieudonn\'e determinant is an isomorphism. One can use these properties of the Dieudonn\'e determinant to deduce the following additional properties:

\begin{enumerate}
\setcounter{enumi}{2}
\item If $A$ is a permutation matrix, then $\det(A)$ is a unit in $T$ (since $\det(A)^2=1$). We say that the matrix $A = (a_{ij})$ is a permutation matrix if there exists distinct indices $r,s$ ($r \neq s$) such that $a_{ij} = \left\{ \begin{array}{cc} 1, \quad & \text{if } i = j \neq r, \ i = j \neq s \\ 1 \quad & \text{ if } i=r, \ j=s, \\  1 \quad & \text{ if } i=s, \ j=r, \\ 0 \quad & \text{otherwise}. \end{array}  \right.$
\item If $A$ is a triangular matrix in $\Gl_n(T)$, then $\det(A)=\prod_{i} a_{ii}$ in $K_1(T)$, where $a_{ii}$'s are the entries on the main diagonal of $A$.
\end{enumerate}

We will consider the following property for the semi-local ring $T$:
\begin{enumerate}[label=$(\mathrm{\mathbf{WP}})$]
\item\label{Wproperty} $\frac{T}{\mathrm{Jac}(T)}$ is a product of matrix rings, none of which is $M_2(\mathbb{F}_2)$ and at most one of these factors is $\mathbb{F}_2$.
\end{enumerate}
Note that if $2$ is invertible in $T$, then the property \ref{Wproperty} holds. Vaserstein \cite{MR2111217} has shown that for a semi-local ring, if property \ref{Wproperty} holds, then we have the following natural isomorphism of abelian groups
\begin{align*} \frac{T^*}{W(T)} \cong (T^*)^{ab}. \end{align*}

Here, $(T^*)^{ab}$ denotes the abelianization of the unit group $T^*$. To provide an illustration, consider a $2 \times 2$ matrix $\left[\begin{array}{cc} a & b \\ c & d\end{array} \right] $ in $\Gl_2(T)$. Also, assume that the element $a$ is invertible in $T$. Then, using properties of determinants described above, one can show that
\begin{align} \label{computation-2-by-2}
\det\left(\left[\begin{array}{cc} a & b \\ c & d\end{array} \right]   \right) & = \det\left(\left[\begin{array}{cc} 1 & 0 \\ ca^{-1} & 1\end{array} \right]\left[\begin{array}{cc} a & 0 \\ 0 & d-ca^{-1}b\end{array} \right]\left[\begin{array}{cc} 1 & a^{-1}b \\ 0 & 1\end{array} \right] \right) \\ & \notag = [ad-aca^{-1}b] \in  \left(T^*\right)^{\ab}.
\end{align}

\begin{example} \label{eg1} Let $\mathbb{H}$ be the quaternion division algebra over the real numbers $\mathbb{R}$. Note that $\mathbb{H}$ is a $4$-dimensional vector space over $\mathbb{R}$ generated by $1$, $i$, $j$ and $k$ satisfying the usual properties:
\begin{align*}
i^2 = j^2 = k^2 =-1, \qquad ij = -ji = k, \qquad jk = -kj = i, \qquad ki = -ik = j.
\end{align*}
Consider the following example of an invertible $2 \times 2$ matrix in $\Gl_2(\mathbb{H})$:
\begin{align*}
& \left[\begin{array}{cc} a & b \\ c & d\end{array} \right]  =\left[\begin{array}{cc} i & j \\ j & i\end{array} \right], \qquad  && \left[\begin{array}{cc} i & j \\ j & i\end{array} \right] \left[\begin{array}{cc} i & j \\ j & i\end{array} \right] =\left[\begin{array}{cc} -2 & 0 \\ 0 & -2\end{array} \right],\\  & ad-aca^{-1}b =-2,  && ad - bc = da -bc = ad -cb = da-cb =0.
\end{align*}
\end{example}
The example above shows that $\det(A)$, in general, is not uniformly represented by $ad-bc$ or $da-bc$ or $ad-cb$ or $da - cb$. Nevertheless, one may still ask the following question:

\begin{question} \label{naive-question}
Let $T$ be a semi-local ring satisfying \ref{Wproperty}.  Let $R \hookrightarrow T$ be a subring. Suppose the $n \times n$ matrix $A$ belongs to $M_n(R) \cap \Gl_n(T)$. Does $\det(A)$ lie in the image of the natural  map
\begin{align*}
i: R \cap T^* \rightarrow (T^*)^{ab},
\end{align*}
of multiplicative monoids?
\end{question}

When $T$ is commutative, Question \ref{naive-question} has an affirmative answer. However, considering the level of generality at which it is phrased, Question \ref{naive-question} has a negative answer. Consider the following example described in Problem 3 in Section 7.10 of Cohn's book \cite{MR2246388}.

\begin{example} \label{eg2}
Let $k$ be a field such that $\text{Char}(k)\neq 2$. Let $R = k<x,y,z,t>$ be the free (non-commutative) algebra in $4$ indeterminates. Let $U_R$ denote its universal skewfield of fractions. See Section 7.2 in \cite{MR2246388} for the definition of universal skewfield of fractions and the properties that this skewfield $U_R$ satisfies. See Corollary 2.5.5 and Corollary 7.5.14 in \cite{MR2246388}, as to why this ring $R$ has a universal skewfield of fractions. For our purposes, we will simply keep in mind that $U_R$ is not obtained via the Ore localization of $R$ at the multiplicatively closed set $R \setminus\{0\}$.  Let $A$ be the $2\times 2$ matrix $\left(\begin{array}{cc}x & y \\ z & t \end{array}\right)$ in $M_2(R) \cap \Gl_2(U_R)$. In this case,  $\det(A) = \big[x\left(t - zx^{-1}y\right)\big]$ inside $\left(U_R^{*}\right)^{\text{ab}}$. However, this element of $\left(U_R^*\right)^{\text{ab}}$ has no representative in $R$. \qed
\end{example}

For our purposes, we would like to refine Question \ref{naive-question} so that the refined question may have an affirmative answer. We will follow some of the terminology given in the book of Goodearl and Warfield  \cite{MR2080008}. Let $S$ be a multiplicatively closed set in a ring $R$. The set $S$ is called a left-reversible left-Ore set if it satisfies the following two conditions:
\begin{enumerate}
\item (left cancellation) If $ns = ms$, for some $n,m$ in $R$ and some $s$ in $S$, then there exists $s'$ in $S$ such that $s'n = s'm$.
\item (left Ore condition) For every $r \in R$ and $s \in S$, there exists $r' \in R$, $s' \in S$ so that $s'r = r's$.
\end{enumerate}

One can similarly define a right-reversible right-Ore set. A multiplicatively closed set $S$ will be called an Ore set if it is both a left-reversible left-Ore set and a right-reversible right-Ore set. If $S$ is an Ore set in a ring $R$, it will be possible to consider the localization $R_S$. The set $S$ is also often called a denominator set and the ring $R_S$ is often called the Ore localization of $R$ at $S$. See Chapter 9 in the book of Goodearl and Warfield \cite{MR2080008}, especially Theorem 9.7 and Proposition 9.8 there.  We ask the following variant of Question \ref{naive-question}.

\begin{question} \label{ore-localization-question-refined}
Let $S$ be an Ore set in a semi-local ring $R$, so that the localization $R_S$ is also a semi-local ring. Suppose $R_S$ satisfies \ref{Wproperty}. Let $A$ be a matrix that belongs to $M_n(R) \cap \Gl_n(R_S)$. When does $\det(A)$ lie in the image of the natural map
\begin{align*}
i:R \cap R_S^* \rightarrow (R_S^*)^{ab}.
\end{align*}
of multiplicative monoids?
\end{question}

To introduce one piece of terminology, we will follow the notations of Question \ref{ore-localization-question-refined}. We will say that Question \ref{ore-localization-question-refined} has a positive answer for the pair $(R, R_S)$ if the following statement is true:
\begin{align*}
\forall n \geq 0, \forall A \in M_n(R) \cap \Gl_n(R_S) \implies \det(A) \in  i\left(R \cap R_S^*\right) \subset (R_S^*)^{ab}.
\end{align*}
Otherwise, we will say that Question \ref{ore-localization-question-refined} has a negative answer for $(R, R_S)$. \\

To introduce another piece of terminology, we will say that a matrix $A$ in $M_n(T)$ admits a \textit{diagonal reduction via elementary operations}, if there exists matrices $U$ and $V$ in $\Gl_n(T)$, and a diagonal matrix $B$ in $M_n(T)$, so that
\begin{enumerate}
\item $A = UBV$,
\item The matrices $U$ and $V$ are obtained as products of matrices of the following kinds:
\begin{itemize}
\item elementary matrices,
\item permutation matrices,
\item scalar matrices in $\Gl_n(T)$.
\end{itemize}
\end{enumerate}

We recall the definition of a principal ideal domain in the non-commutative setting given in Jacobson's book \cite{MR0008601}. A (not necessarily commutative) domain is said to be a principal left ideal domain if every left ideal is principal. A domain is said to be a principal right ideal domain if every right ideal is principal. A domain is said to be a principal ideal domain if it is both a principal left ideal domain and a principal right ideal domain.
\begin{proposition}[Theorem 16, Chapter 3 in \cite{MR0008601}] \label{diagonal-reduction}
Let $T$ be a principal ideal domain. Every matrix $A$ in $M_n(T)$ admits a diagonal reduction via elementary operations.
\end{proposition}

Let $O$ be a complete discrete valuation ring, whose fraction field is denoted by $K$. Let $D$ be a division algebra whose center contains $K$ and such that the index $[D:K]$ is finite. Let $O_D$ be the maximal $O$-order inside $D$. In this case, $O_D$ is a (non-commutative) principal ideal domain. See Theorem 13.2 in Reiner's book on Maximal orders \cite{MR0393100} where it is established that $O_D$ is a principal ideal domain. Every $n \times n$ matrix with entries in $O_D$ admits a diagonal reduction via elementary operations. See also Theorem 17.7 in \cite{MR0393100} for this fact. In this case, Question \ref{ore-localization-question-refined} has a positive answer for the pair $(O_D,D)$. \\

Unfortunately, we will not be able to classify the tuples $(R,R_S)$ for which Question \ref{ore-localization-question-refined} has a positive answer. Nevertheless, we will   provide one example (Example \ref{dihedral-example}) when Question \ref{ore-localization-question-refined} has a positive answer and one example (Example \ref{quaternion-example}) when Question \ref{ore-localization-question-refined} has a negative answer.

\begin{example} \label{dihedral-example}
Let $p$ be an odd prime. In this example, we shall show that Question \ref{ore-localization-question-refined} has a positive answer for the pair $(\Z_p[D_{2p}], \Q_p[D_{2p}])$. Let $D_{2p}$ be the dihedral group of order $2p$ which has the following presentation
\begin{align*}
D_{2p} =\{x,y | x^2 = y^p =1 , \ xyx^{-1} = y^{-1}\}.
\end{align*}
Note that since $p$ is odd, the ring $\Z_p[D_{2p}]$ satisfies \ref{Wproperty}.

Let $L = \Q_p(\zeta_p)$. Let $F =\Q_p(\zeta_p + \zeta_p^{-1})$. Here, we let $\zeta_p$ denote a primitive $p$-th root of unity. The Galois group $\Gal{L}{F}$ is of order two. Let $\alpha$ denote the non-trivial element in $\Gal{L}{F}$. Let $O_{L}$ and $O_{F}$ denote the ring of integers in $L$ and $F$ respectively. Let $\p_L$ and $\p_F$ denote the unique prime lying above $p$ in $O_L$ and $O_{F}$ respectively.  We have the following equality:
\begin{align*}
\p_L= (1-\zeta_p) \text{ as ideals in } O_{L}, \qquad \p_F= (2-\zeta_p - \zeta^{-1}_p) \text{ as ideals in } O_{F}.
\end{align*}
Since the field extensions $L/\Q_p$ and $F/\Q_p$ are totally ramified, we have the following natural isomorphisms of residue fields:
\begin{align} \label{totally-ramified-residue-fields}
\frac{\Z}{p\Z} \stackrel{\cong}{\hookrightarrow} \frac{O_F}{\p_F} \stackrel{\cong}{\hookrightarrow} \frac{O_L}{\p_L}.
\end{align}

Note that the Artin-Wedderburn theorem gives us an isomorphism of $\Q_p$-algebras:
\begin{align} \label{artin-wedderburn-iso-dihedral}
\Q_p[D_{2p}] \cong \Q_p[C_2] \times L<\tau>.
\end{align}
Here, $C_2$ is a cyclic group of order 2 with generator $e$. The central simple $F$-algebra $L<\tau>$ is given by $L \oplus L \cdot \tau$, where we have $\tau^2=1$ and $\tau a = \alpha(a) \tau$, for all $a \in L$. The center of $L<\tau>$  equals $F$. Note that since $(\tau-1)(\tau+1)=0$, the simple algebra $L<\tau>$ cannot be a division algebra. Also, $\dim_{F} L<\tau>=4$. By a simple dimension counting argument and the Artin-Wedderburn theorem, one can see that we have the isomorphism $L<\tau> \cong M_2(F)$ of $F$-algebras. However, it will be convenient to view $L<\tau>$ naturally inside $L<\tau> \otimes_F L $ since this allows us to consider the isomorphism $L<\tau> \otimes_F L \cong M_2(L)$.  One can obtain such an isomorphism by considering the following assignments:
\begin{align} \label{embedding-into-matrices}
\zeta_p \rightarrow \left[\begin{array}{cc} \zeta_p & 0 \\ 0 & \zeta_p^{-1} \end{array} \right], \quad \tau \rightarrow \left[\begin{array}{cc} 0 & 1 \\ 1 & 0\end{array}\right].
\end{align}

The isomorphism in (\ref{artin-wedderburn-iso-dihedral}) is chosen to agree with the following two projection maps: \begin{align*}
& \sigma_1: \Q_p[D_{2p}] \rightarrow \Q_p[C_2], \qquad \qquad && \sigma_2: \Q_p[D_{2p}] \rightarrow L<\tau>. \\
& \sigma_1(x) =e, \ \sigma_1(y) =1 && \sigma_2(x) =\tau,  \ \sigma_2(y) = \zeta_p.
\end{align*}

Let $n$ denote a positive integer. Let us label $\det$, $\Det_1$ and $\Det_2$ for the Dieudonn\'e determinants involving the invertible matrices in $\Gl_n(\Q_p[D_{2p}])$, $\Gl_n(\Q_p[C_2])$ and $\Gl_n(L<\tau>)$ respectively. Note that the reduced norm $\mathrm{Nrd}:L<\tau>^* \rightarrow F^*$ is given by the formula $\mathrm{Nrd}(c +d \tau ) = c \alpha(c) - d \alpha(d)$. The reduced norm gives us an isomorphism $\mathrm{Nrd}:(L<\tau>^*)^{ab} \xrightarrow {\cong} F^*$. See Theorem 2.3 in Oliver's book \cite{MR933091}. We have the following diagram relating these  Dieudonn\'e determinants:
\begin{align*}
\xymatrix{
\Gl_n(\Q_p[D_{2p}])  \ar[rd]^{\det} \ar[r]^{\cong \qquad \qquad } &  \Gl_n(\Q_p[C_2]) \times \Gl_n(L<\tau>) \ar[d]^{(\det_1, \det_2)} \\
& \Q_p[C_2]^* \times (L<\tau>^*)^{ab} \ar[r]^{\qquad \mathrm{Nrd}}_{\qquad \cong}
& \Q_p[C_2]^* \times F^*
}
\end{align*}

We will follow the description of the integral group ring $\Z_p[D_{2p}]$, i.e. its image under the isomorphism (\ref{artin-wedderburn-iso-dihedral}), given in Section 8 of the work of Reiner and Ullom \cite{MR0304470}.
 Let $O_{L<\tau>}$ denote the subring $O_{L} \oplus O_{L}\cdot \tau$ of the central simple $F$-algebra $L<\tau>$.  This is a maximal $\Z_p$-order inside $L<\tau>$. Under the isomorphism (\ref{artin-wedderburn-iso-dihedral}), we have
\begin{align}
\Z_p[D_{2p}] \cong \{ (a + be, c+d\tau ) \in \Z_p[C_2] \times O_{L<\tau>}, \text{ such that } a \equiv c \ (\mathrm{mod} \ \p_L), b \equiv d \ (\mathrm{mod} \ \p_L) \}.
\end{align}

The ring $\Z_p[D_{2p}]$ is a semi-local ring whose Jacobson radical is given below:
\begin{align*}
\mathrm{Jac}(\Z_p[D_{2p}]) = \ker\left( \Z_p[D_{2p}] \rightarrow \mathbb{F}_p[C_2]\right).
\end{align*}
Let $\m$ denote the ideal $(p,y-1,x-1)$ and let $\m'$ denote the ideal $(p,y-1,x+1)$ in $\Z_p[D_{2p}]$. The ideals $\m$ and $\m'$ are both left-maximal and right-maximal ideals. An element $u$ in $\Z_p[D_{2p}]$, that does not belong to both $\m$ and $\m'$, must be a unit in the ring $\Z_p[D_{2p}]$. See Theorem 1.11 in Oliver's book \cite{MR933091}.

The description of the Jacboson radical of $\Z_p[C_2]$ is given below:
\begin{align*}
 \mathrm{Jac}(\Z_p[C_2]) &= \{a+be \in \Z_p[C_2], \text{ such that } a \in p\Z_p, \ b \in p \Z_p\}.
 \end{align*}

We would like to record three observations.

\begin{enumerate}
\item Suppose we are given an element $a + be \in \mathrm{Jac}(\Z_p[C_2])$, where $a,b \in p\Z_p$. Suppose also that we are given an element $\varpi$ in the maximal ideal $\p_F$ of the ring $O_{F}$. It will be possible to write $\varpi$ as $(2-\zeta_p - \zeta^{-1}_p)^n v$ for some positive integer $n$ and some unit $v$ of the ring $O_F$. It is easy to see that $\mathrm{Nrd}(1-\zeta_p)=2-\zeta_p -\zeta_p^{-1}$. Also, the restriction of the reduced norm $\mathrm{Nrd}: O_{L<\tau>}^\times \twoheadrightarrow O_{F}^\times$ is surjective on the units. See Theorem 2.3 in Oliver's book \cite{MR933091}. This lets us find a unit $u$ in the ring $O_{L<\tau>}$ such that $\mathrm{Nrd}(u)=  v$. Our observations allow us to make the following deduction:
\begin{align} \label{everyelement-determinant}
\epsilon = \bigg(a+be, \  (1-\zeta_p)^n u \bigg) \in \Z_p[D_{2p}], \qquad \mathrm{Nrd}\big((1-\zeta_p)^nu  \big) = \varpi.
\end{align}

\item Suppose now we are given an element $a + be \in \Z_p[C_2]$ such that (i) both $a$ and $b$ belong to $\Z_p^\times$ and such that (ii) $a - b \equiv 0 \ (\mathrm{mod }\ p) $ or $a + b \equiv 0 \ (\mathrm{mod }\ p) $. Suppose also that we are given an element $\varpi$ in the maximal ideal $\p_F$ of the ring $O_{F}$. The restriction of the reduced norm map $\mathrm{Nrd}: 1 + \p_L \twoheadrightarrow 1 + \p_F$ is surjective. See Chapter 1, Section 8, Proposition 2 of Fr\"ohlich's article on local fields \cite{MR0236145} as to why the reduced norm map is surjective on the group of principal units for tamely ramified extensions. So, it will be possible to find an element $u_1 \in 1 + \p_L $ such that $\mathrm{Nrd}(u_1)  = 1 + \frac{\varpi}{a^2}$. Set
\begin{align*}
z = \left\{ \begin{array}{cc} au_1 + a\tau, & \text{ if } a - b \equiv 0 \ (\mathrm{mod }\ p) \\ au_1 - a\tau, & \text{ if } a + b \equiv 0 \ (\mathrm{mod }\ p)  \end{array} \right.
\end{align*}

It is then straightforward to check that
\begin{align} \label{oneelement-determinant}
\epsilon= \bigg(a+be, \  z  \bigg) \in \Z_p[D_{2p}], \qquad \mathrm{Nrd}\left(z \right) = \varpi.
\end{align}

\item Suppose we have two elements $a$ and $a'$ in the ring $\Z_p[D_{2p}]$ such that \begin{align*}
& a \notin \m,\quad  && a \in \m',  \\ & a' \in \m,\quad && a' \notin \m'.
\end{align*}
Then, $a+a' \notin \m$ and $a + a' \notin \m'$. In particular, $a+a'$ is a unit in the ring $\Z_p[D_{2p}]$.
\end{enumerate}

Let $A \in M_{n}(\Z_p[D_{2p}]) \cap \Gl_{n}(\Z_p[D_{2p}])$. We shall show that we can find a representative for $\det(A)$ in $M_n(\Z_p[D_{2p}])$. \\

First, we will consider the case when $A$ belongs to $M_n(\m\Z_p[D_{2p}])$ or $M_n(\m'\Z_p[D_{2p}])$. Without loss of generality, we shall assume that $A$ belongs to $M_n(\m'\Z_p[D_{2p}])$. The argument proceeds similarly when $A$ belongs to $M_n(\m\Z_p[D_{2p}])$. Note that the ring $\Q_p[C_2]$ is commutative. The matrix $\sigma_1(A)$ is an $n \times n$ matrix with entries in $\Z_p[C_2]$. Let us write $\Det_1(\sigma_1(A))$  as $a+be$. Since we have assumed that $A$ belongs to $M_n(\m'\Z_p[D_{2p}])$, every entry in the matrix $\sigma_1(A)$ must belong to ideal $(p,e+1)$. So, $p$ must divide $a - b$. As a result, one sees that in this first case, we have \begin{align*}
(1) \ a + be \in \mathrm{Jac}(\Z_p[C_2]), \text{ or} \qquad  (2) \ a ,b \in \Z_p^\times \text{ such that }a -b \equiv 0 \ \mathrm{mod} \ p.
\end{align*}
The matrix $\sigma_2(A)$ is an $n \times n$ matrix with entries in $O_{L<\tau>}$. One can check that the $(i,j)$-th entry of $\sigma_2(A)$ can be written as $a_{ij}(\tau+1) + b_{ij} p + c_{ij}(\zeta_p-1)$, for some elements $a_{ij}$, $b_{ij}$ and $c_{ij}$ in $O_{L<\tau>}$.  We will use the assignments given in (\ref{embedding-into-matrices}) to fix an embedding $\mathfrak{i}: L<\tau> \hookrightarrow M_2(L)$. It will be possible to view the matrix $\mathfrak{i}(\sigma_2(A))$ as a matrix in $M_{2n}(O_L)$.  Note that $\zeta_p-1$ and $p$ belong to the ideal $\p_L O_L$. The assignment given in (\ref{embedding-into-matrices}) sends $\tau +1$ to the $2\times 2$ matrix $\left[\begin{array}{cc} 1 & 1 \\ 1 & 1\end{array}\right]$. We have the following equality of $2n\times 2n$ matrices modulo $\p_L O_L$:
\begin{align*}
\mathfrak{i}(\sigma_2(A))  \equiv \left[\begin{array}{ccc} \mathfrak{i}(a_{11}(\tau+1))  &  \cdots & \mathfrak{i}(a_{1n}(\tau+1))   \\ \vdots &  \ddots  & \vdots \\ \mathfrak{i}(a_{n1}(\tau+1))  & \cdots & \mathfrak{i}(a_{nn}(\tau+1)) \end{array} \right]    \equiv \left[\begin{array}{ccc} \mathfrak{i}(a_{11})  &  \cdots & \mathfrak{i}(a_{1n})   \\ \vdots &  \ddots  & \vdots \\ \mathfrak{i}(a_{n1})  & \cdots & \mathfrak{i}(a_{nn}) \end{array} \right] \left[\begin{array}{cccccccc} 1  & 1 & 0 & 0  & \cdots & 0  &0   \\  1  & 1 & 0 & 0  & \cdots & 0 & 0 \\ 0  & 0 & 1 & 1  & \cdots & 0  &0   \\  0  & 0 & 1 & 1  & \cdots & 0  &0   \\ & & & & \ddots & & \\ 0  & 0 & 0 & 0 & \cdots & 1 & 1 \\  0 &0 & 0 & 0 & \cdots & 1  & 1 \end{array} \right].
\end{align*}
It is then easy to see that \begin{align}\label{intermediate-det-field}
\Det(\mathfrak{i}(\sigma_2(A))) \equiv 0 \ (\mathrm{ mod }\ \p_L).
\end{align}
The $\Det$ in (\ref{intermediate-det-field}) involves the determinant, over the commutative field $L$, of the $2n \times 2n$ matrix $\mathfrak{i}(\sigma_2(A))$. This lets us conclude that $\mathrm{Nrd}(\sigma_2(A))$ belongs to $\p_F O_F$ (since $O_F \cap \p_L = \p_F$). If we let $\varpi$ denote $\mathrm{Nrd}(\sigma_2(A))$, we have
$$\mathrm{Nrd}(\sigma_2(A))  = \varpi \in \p_F O_F.$$
In this first case, our earlier observations in equation (\ref{everyelement-determinant}) and equation (\ref{oneelement-determinant}) allow us to find an element $\epsilon$ in $\Z_p[D_{2p}]$ such that $\det([\epsilon])=\det(A)$. \\

In the second case, we shall suppose that there exists entries $a_{i,j}$ and $a_{i',j'}$ of the matrix $A$ such that $a_{i,j} \notin \m$ and $a_{i',j'} \notin \m'$. It is then straightforward (but slightly tedious) to see that one can perform a sequence of elementary row and column operations on the matrix $A$ to find a new matrix, one of whose entries lies in neither $\m$ nor $\m'$. Such an entry must be a unit in the ring $\Z_p[D_{2p}]$.  Since the elementary row and column operations do not change the Dieudonn\'e determinant, this new matrix would have the same Dieudonn\'e determinant as the matrix $A$. So without loss of generality, in this second case, we can suppose that there exists an entry $u$, in the $n \times n$ matrix $A$,  which is a unit in the ring $\Z_p[D_{2p}]$. One can perform elementary row and column operations (similar to the ones used to obtain the formula in (\ref{computation-2-by-2})) and use permutation matrices to obtain an $n \times n$ matrix $B_n$ so that \begin{align*}
B_n :=\left[\begin{array}{cc} u & 0 \\ 0 & B_{n-1}\end{array} \right], \text{ such that } B_{n-1} \in M_{n-1}\left(\Z_p[D_{2p}]) \cap \Gl_{n-1}(\Z_p[D_{2p}]\right), \quad \det(A) = \det(B_n).
\end{align*}
We have $\det(A)$ equals $u \det(B_{n-1})$ as elements of $\left(\Q_p[D_{2p}]^{*}\right)^{ab}$. One can now use mathematical induction to conclude that Question \ref{ore-localization-question-refined} has a positive answer for the pair $(\Z_p[D_{2p}],\Q_p[D_{2p}])$. \qed
\end{example}

\begin{example} \label{quaternion-example}
Let $H_8$ denote the group of quaternions. This group has 8 elements given by the following presentation:
\begin{align*}
H_8 = \{x,y |x^4=1, x^2= y^2 , yxy^{-1} = x^{-1}\}
\end{align*}
Note that $\Z_2[H_8]$ is a (non-commutative) local ring and that the quotient $\frac{\Z_2[H_8]}{\mathrm{Jac}(\Z_2[H_8])}$ is isomorphic to $\mathbb{F}_2$. The ring $\Z_2[H_8]$ satisfies \ref{Wproperty}. Question \ref{ore-localization-question-refined} has a negative answer for the pair $(\Z_2[H_8],\Q_2[H_8])$. In fact, we will give an example of a matrix $A$ in $M_2(\Z_2[H_8]) \cap \Gl_2(\Q_2[H_8])$, such that under the Dieudonn\'e determinant $\det: \Gl_2(\Q_2[H_8]) \rightarrow \left(\Q_2[H_8]^{*}\right)^{ab}$, the determinant $\det(A)$ does not lie in the image of the map $\Z_2[H_8] \cap \Q_2[H_8]^* \rightarrow \left(\Q_2[H_8]^{*}\right)^{ab}$. \\

Note that the Artin-Wedderburn theorem gives us an isomorphism of $\Q_2$-algebras:
\begin{align} \label{artin-wedderburn-iso}
\Q_2[H_8] \cong \Q_2[C_2 \oplus C_2] \times D.
\end{align}
Here, $D$ is the division algebra of rational quaternions given by $$ D = \Q_2 \oplus \Q_2 i \oplus \Q_2 j \oplus \Q_2 ij.$$ We will write the Klein-four group $C_2 \oplus C_2$ as $\{e,f | e^2 = f^2 =1, ef = fe\}$, so that
\begin{align*} \Q_2[C_2 \oplus C_2] = \Q_2 \oplus \Q_2 e \oplus \Q_2f \oplus \Q_2 ef.
\end{align*}
The isomorphism in (\ref{artin-wedderburn-iso}) is chosen to agree with the following two projection maps: \begin{align*}
& \sigma_1: \Q_2[H_8] \rightarrow \Q_2[C_2 \oplus C_2], \qquad \qquad && \sigma_2: \Q_2[H_8] \rightarrow D. \\
& \sigma_1(x) =e, \sigma_1(y) =f && \sigma_2(x) =i, \sigma_2(y) = j.
\end{align*}

Let us label $\det$, $\Det_1$ and $\Det_2$ for the Dieudonn\'e determinants involving the invertible matrices in $\Gl_2(\Q_2[H_8])$, $\Gl_2(\Q_2[C_2 \oplus C_2])$ and $\Gl_2(D)$ respectively. Note that the reduced norm $\mathrm{Nrd}:D^* \rightarrow \Q_2^*$ is given by the formula $\mathrm{Nrd}(b_1 + b_2i + b_3j + b_4ij ) = b_1^2 + b_2^2 + b_3^2 + b_4^2$. The reduced norm gives us an isomorphism $\mathrm{Nrd}:(D^*)^{ab} \rightarrow \Q_2^*$. See Theorem 2.3 in Oliver's book \cite{MR933091}. We have the following diagram relating these  Dieudonn\'e determinants:
\begin{align*}
\xymatrix{
\Gl_2(\Q_2[H_8])  \ar[rd]^{\det} \ar[r]^{\cong \qquad \qquad } &  \Gl_2(\Q_2[C_2 \oplus C_2]) \times \Gl_2(D) \ar[d]^{(\det_1, \det_2)} \\
& \Q_2[C_2 \oplus C_2]^* \times (D^*)^{ab} \ar[r]^{\mathrm{Nrd}}_{\cong}
& \Q_2[C_2 \oplus C_2]^* \times \Q_2^*
}
\end{align*}

We will follow the description of the integral group ring given in Section 7b of the work of Reiner and Ullom \cite{MR0349828}. Let $Z_{D} :=\Z_2 \oplus \Z_2i\oplus \Z_2j \oplus \Z_2ij$. Reiner and Ullom identify $\Z_2[H_8]$ with the following subring of $\Q_2[C_2 \oplus C_2] \times D$ under the isomorphism given in (\ref{artin-wedderburn-iso}) :
\begin{align} \label{integral-quaternions}
\Z_2[H_8] \cong \bigg\{ & \bigg(a_1 + a_2 e + a_3 f  + a_4 ef , b_1 + b_2 i + b_3 j + b_4 ij\bigg) \in \Z_2[C_2 \oplus C_2] \times Z_{D}, \\  & \notag \text{ such that } (a_1,a_2,a_3,a_4) \equiv (b_1,b_2,b_3,b_4) \in \mathbb{F}_2 [C_2 \oplus C_2] \bigg\}.
\end{align}
Note that $\Z_2[H_8]$ is a local ring, with a unique maximal left ideal given below:
\begin{align*}
\mathrm{Jac}(\Z_2[H_8]) =\big\{ & a_0 + a_1 x + a_2 x^2 + a_3 x^3 + a_4 y + a_5 xy + a_6 x^2y + a_7 x^3y \in  \Z_2[H_8] \\ & \text{ such that } a_0 + a_1  + a_2 + a_3 + a_4  + a_5  + a_6  + a_7 \in 2\Z_2 \big\}.
\end{align*}
Suppose we have a matrix  in $M_2(\Z_2[H_8] \cap \Gl_2(\Q_2[H_8])$. If one of the non-zero elements in this matrix belongs to the center of $\Q_2[H_8]$ or is a unit in $\Z_2[H_8]$, one can use a formula analogous to the one in (\ref{computation-2-by-2}), to show that the Dieudonn\'e determinant of this matrix does  have a representative in $\Z_2[H_8]$. \\

Let $A$ denote the $2 \times 2$ matrix $ \left[\begin{array}{cc} 9 + x + 2y & 1 + y \\ 1 + xy & 9+x\end{array} \right]$. Note that every element of this matrix belongs to the maximal ideal of $\Z_2[H_8]$. Note also that both $1+y$ and $1+xy$ do not belong to the centralizer of $9+x+2y$ or to the centralizer of $9+x$. The field $\Q_2(\sqrt{-1})$ is a splitting field for $D$. That is, we have an isomorphism $D \otimes_{\Q_2} \Q_2(\sqrt{-1}) \cong M_2(\Q_2(\sqrt{-1}))$ obtained by the following assignments:
\begin{align*}
x \rightarrow \left[\begin{array}{cc} \sqrt{-1} & 0 \\ 0 & -\sqrt{-1}  \end{array}\right], \quad y \rightarrow \left[\begin{array}{cc} 0 &  1 \\ -1 & 0 \end{array}\right], \quad xy \rightarrow \left[\begin{array}{cc} 0 & \sqrt{-1}  \\ \sqrt{-1}  & 0 \end{array}\right].
\end{align*}
A direct computation then gives us the following equalities:
\begin{align*}
 \Det_1(\sigma_1(A)) &= 81 + 17e + 17f + ef \in \Q_2[C_2\ \oplus C_2]^*, \\
 \mathrm{Nrd}\left(\Det_2(\sigma_2(A)) \right) &= \det\left[\begin{array}{cccc}9+\sqrt{-1} & 2  & 1 & 1 \\ -2 & 9 - \sqrt{-1} & -1 & 1 \\  1 & \sqrt{-1} & 9 + \sqrt{-1} & 0 \\ \sqrt{-1} & 1  & 0 & 9 - \sqrt{-1} \end{array} \right]=  8 \times 857 \in \Q_2^*.
\end{align*}
Recall that if $$a \equiv 1 \mod 2\Z_2 \implies a^2 \equiv 1 \mod 8\Z_2.$$ Let us use this observation along with the isomorphism in (\ref{integral-quaternions}). Suppose that the element $$\left(81 + 17e + 17f + ef , b_1 + b_2 i + b_3 j + b_4 ij\right) \in \Z_2[C_2 \oplus C_2] \times Z_{D}$$ belongs to the subring $\Z_2[H_8]$. Then, $\mathrm{Nrd}(b_1 + b_2 i + b_3 j + b_4 ij) \equiv 4 \mod 8\Z_2$ (in particular, the reduced norm is not divisible by  $8$). This shows us the Dieudonn\'e determinant $\det(A)$, for the $2\times 2$ matrix $A$ given in this example does not have any representative in the integral group $\Z_2[H_8]$.
\qed
\end{example}

\subsection{A maximal $\Lambda$-order} \label{sec:maxorder} \mbox{}

The Artin-Wedderburn theorem gives us the following isomorphism of $\Q_p$-algebras:
\begin{align}\label{AW-groupring}
\Q_p[G] \cong \prod_i M_{m_i}(D_i).
\end{align}
Here, $D_i$ is a finite-dimensional division algebra over $\Q_p$. Let $O_{D_i}$ denote the unique maximal $\Z_p$-order inside $D_i$.
Note that any $\Z_p$-order in $\Q_p[G]$ can be embedded in a maximal order and any two maximal orders in $M_n(D_i)$ are isomorphic via an inner automorphism of $M_n(D_i)$. This allows us to choose the isomorphism in (\ref{AW-groupring}) so that the following diagram commutes:
\begin{align}
\xymatrix{
\Z_p[G] \ar[r] \ar[d] & \Q_p[G] \ar[d]^{\cong} \\ \prod_i M_{m_i}\left(O_{D_i}\right) \ar[r]& \prod_i M_{m_i}(D_i).
}
\end{align}
The horizontal maps are the natural injections. The vertical map on the left is also injective. This induces the commutative diagram given below:
\begin{align} \label{comm-diag-artin}
\xymatrix{
\Lambda[G] \ar[r] \ar[d] & Q_\Lambda[G] \ar[d]^{\cong} \\ \prod_i M_{m_i}\left(O_{D_i} \otimes_{\Z_p} \Lambda \right) \ar[r]& \prod_i M_{m_i}(D_i\otimes_{Q_p} Q_\Lambda).
}
\end{align}
Here, $Q_\Lambda$ denotes the fraction field of $\Lambda$. Once again, the horizontal maps are injective. The vertical map on the left is also injective. \\

For the rest of Section \ref{section-non-commutative}, we let $D$ denote a divison ring containing $\Q_p$ inside its center and such that $[D:\Q_p]$ is finite. We let $F$ denote the center of $D$, whose ring of integers is denoted by $O_F$. Let $L$ denote (unfortunately, in other sections, the letter $L$ has been used in another context. In this section, and only in this section, we use the letter $L$ to denote the maximal subfield of the division algebra $D$) a maximal subfield of $D$ containing $F$. Let $O_L$ denote the ring of integers in $L$. The fields $F$ and $L$ are finite extensions of $\Q_p$. We let $O_D$ denote the unique maximal $\Z_p$-order inside $D$. We recall some of the properties of $O_D$ from Reiner's book on maximal orders \cite{MR0393100}:

\begin{enumerate}
\item $O_D$ is the integral closure of $\Z_p$ in $D$. See Theorem 12.8 in Reiner's book \cite{MR0393100}.
\item There exists a discrete valuation $w$ on $D$, extending the $p$-adic valuation on $\Z_p$. The ring $O_D$ is the valuation ring, with respect to $w$, inside $D$.  See Chapter 12 in Reiner's book \cite{MR0393100}. We let $\pi_D$ denote a uniformizer in $O_D$, for this valuation.
\end{enumerate}

We will also use the following notations:
\begin{align*}
& F_\Lambda:= F \otimes_{\Q_p} Q_\Lambda, \quad O_{F_\Lambda} := O_F \otimes_{\Z_p} \Lambda, \quad L_\Lambda:= L \otimes_{\Q_p} Q_\Lambda, \quad O_{L_\Lambda} := O_L \otimes_{\Z_p} \Lambda
\\ & D_\Lambda := \underbrace{D \otimes_{F}  F_\Lambda}_{\cong D \otimes_{\Q_p} Q_\Lambda}, \quad O_{D_\Lambda} :=O_D \otimes_{\Z_p} \Lambda
\end{align*}
We have the following equalities of vector space dimensions (see Theorem 7.15 in Reiner's book \cite{MR0393100}):
\begin{align} \label{index-d}
\sqrt{\dim_{F_\Lambda} D_\Lambda} = \sqrt{\dim_F D} = \dim_F L = \dim_{F_\Lambda} L_\Lambda = \dim_{L} D = \dim_{L_\Lambda}D_\Lambda.
\end{align}

The number, that is equal to all the quantities appearing in (\ref{index-d}), is called the index of the division algebra $D$ in the Brauer group $\mathrm{Br}(F)$.  As we will shall show in Lemma \ref{establish-division-ring}, $D_\Lambda$ is a division algebra with center $F_\Lambda$. The number appearing in (\ref{index-d}), is also the index of the division algebra $D_\Lambda$ in the Brauer group $\mathrm{Br}(F_\Lambda)$.

One can obtain a non-canonical isomorphism $\Lambda \cong \Z_p[[x]]$ of topological rings, by sending a topological generator  $\gamma_0$ of the topological group $\Gamma$ to the element $x+1$ in $\Z_p[[x]]$. Since $O_D$ has finite rank as a $\Z_p$-module, we have the following lemma:
\begin{lemma} \label{power-series}
The isomorphism $\Lambda \cong \Z_p[[x]]$ of topological rings, obtained by sending a topological generator $\gamma_0$ of $\Gamma$ to $x+1$ $\Lambda \cong \Z_p[[x]]$, lets us obtain the following isomorphisms:
\begin{align*}
\Lambda \cong \Z_p[[x]], \quad O_{F_\Lambda} \cong O_F[[x]], \quad O_{L_\Lambda} \cong O_L[[x]], \quad O_{D_\Lambda} \cong O_D[[x]].
\end{align*}
\end{lemma}

\begin{lemma} \label{global-dim}
$O_{D_\Lambda}$ is a local ring with global dimension equal to two.
\end{lemma}

\begin{proof}
Let $\mathrm{gldim}$ denote global dimension. The power series ring $O_D[[x]]$ in one variable $x$ over the valuation ring $O_D$ is local. By Theorem 2.3 and Proposition 2.7 in \cite{MR0117252}, the global dimension $\mathrm{gldim}\left(O_D\right)$ equals one. Note that by Theorem 7.5.3 in \cite{MR1811901}, we have
\begin{align*} \mathrm{gldim}\left(O_D[[x]] \right)= \mathrm{gldim}\left(O_D\right) + 1 = 2.
\end{align*}
\end{proof}

\begin{lemma} \label{establish-division-ring}
$D_\Lambda$ is a divison ring with center $F_\Lambda$, satisfying property \ref{Wproperty}. $L_\Lambda$ is a maximal commutative subfield, inside $D_\Lambda$, containing $F_\Lambda$.
\end{lemma}

\begin{proof}
Note that $Q_\Lambda$ is a purely transcendental extension of $\Q_p$, while $F$ is a finite (algebraic) extension of $\Q_p$. As a result, $F_\Lambda$ is a field. Since $D$ is a finite-dimensional central simple $F$-algebra, $D_\Lambda$ is a finite-dimensional central simple $F_\Lambda$-algebra (Corollary 7.8 in Reiner's book \cite{MR0393100}). A finite-dimensional central simple algebra over the field $F_\Lambda$ is isomorphic to a matrix ring over a division algebra. To prove the lemma, it suffices to show that $D_\Lambda$ is a (not necessarily commutative) domain.

For this, observe that $D_\Lambda$ is  the localization of $O_D \otimes_{O_F} O_F[[x]]$ at the multiplicatively closed set $O_F[[x]] \setminus \{0\}$. The set $O_F[[x]] \setminus \{0\} $ is central in $D_\Lambda$ and has no zero-divisors. Now, to prove the lemma, we are reduced to showing that $O_D \otimes_{O_F} O_F[[x]]$ is a (not necessarily commutative) domain. This follows simply because $O_{D} \otimes_{O_F} O_F[[x]]$, as observed in Lemma \ref{power-series}, is isomorphic to the power series ring $O_D[[x]]$ in one variable $x$ over the domain $O_D$, and is hence a (not necessarily commutative) domain.

$2$ is invertible in the division algebra $D_\Lambda$. As a  result, property \ref{Wproperty} holds.

An argument, similar to the one above, shows that $L_\Lambda$ is a field. Note that $L_\Lambda$ splits $D_\Lambda$ since
\begin{align*}
D_\Lambda \otimes_{F_\Lambda} L_\Lambda \cong (D \otimes_F F_\Lambda) \otimes_{F_\Lambda} (F_\Lambda \otimes_F L) \cong (D \otimes_F L ) \otimes_{F} F_\Lambda \cong M_d(L) \otimes_F F_\Lambda \cong M_d(L_\Lambda).
\end{align*}
Here, we let $d$ equal $\dim_{F_\Lambda} L_\Lambda$.  The relationship between various vector space dimensions given in (\ref{index-d}), along with Corollary 28.10 in Reiner's book \cite{MR0393100}, lets us conclude that $L_\Lambda$ is a maximal commutative subfield, inside $D_\Lambda$, containing $F_\Lambda$.
\end{proof}

\begin{proposition} \label{establish-max-order}
For every integer $m$, the matrix ring $M_m\left(O_{D_\Lambda}\right)$ is a maximal $\Lambda$-order inside $M_m\left(D_\Lambda\right)$. Furthermore, every maximal $\Lambda$-order inside $M_m\left(D_\Lambda\right)$ is isomorphic to $M_m(O_{D_\Lambda})$ by an inner automorphism via a unit in $M_m\left(D_\Lambda\right)$.
\end{proposition}

\begin{proof}
The proposition follows from  Ramras's work on maximal orders over regular local rings of dimension two. See Theorem 5.4 in \cite{MR0245572}. To verify the hypotheses of Ramras's theorem, we need to show that for each integer $m$, the matrix ring $M_m\left(O_{D_\Lambda}\right)$ is a semi-local ring with global dimension equal to two. Note that since $O_{D_\Lambda}$ is a local ring, the matrix ring $M_m\left(O_{D_\Lambda}\right)$ is semi-local (see 20.4 in Lam's book \cite{MR1838439}). Furthermore, the global dimension is a Morita invariant (see the Proposition in 3.5.10 in \cite{MR1811901}). Note that by Theorem 1.12 in Ramras's work \cite{MR0245572}, for a semi-local ring that is finitely generated over a commutative Noetherian ring, the left and right global dimensions coincide. By Lemma \ref{global-dim}, the global dimension of the matrix ring $M_m\left(O_{D_\Lambda}\right)$ is equal to two, for each integer n. The proposition follows.
\end{proof}

\begin{remark}
In our situation, we are considering maximal orders over the ring $\Lambda$, which is a complete regular local ring of dimension two.  One can manufacture (non-commutative) examples when $O_{D_\Lambda}$ is not the unique maximal $\Lambda$-order inside the division algebra $D_\Lambda$. Contrast this with the fact that $O_D$ is the unique maximal $\Z_p$-order inside the division algebra $D$.
\end{remark}

\subsection{Reduced Norms} \mbox{}

Let us first recall the definition of reduced norms in a general setting. Let $\mathcal{D}$ be a division algebra, finite dimensional over its center $\mathcal{F}$. Let $\mathcal{L}$ denote a maximal subfield of $\mathcal{D}$ containing $\mathcal{F}$. Let $d$ equal the vector space dimension $\dim_{\mathcal{F}} \mathcal{L}$. The field $\mathcal{L}$ is a splitting field for $\mathcal{D}$. That is, we have  $\mathcal{D} \otimes_\mathcal{F} \mathcal{L} \cong M_d(\mathcal{L})$. Consider the inclusion induced by the above isomorphism:
\begin{align}\label{inclusion-splitting-field}
\mathfrak{i} : \mathcal{D} \hookrightarrow  \mathcal{D} \otimes_\mathcal{F} \mathcal{L} \cong M_d(\mathcal{L})
\end{align}
 We have a group homomorphism $\mathrm{Nrd}: K_1(\mathcal{D}) \rightarrow \mathcal{F}^\times$ called the reduced norm map. To recall the definition of the reduced norm, let $\mathcal{A}$ denote a matrix in $\Gl_n(\mathcal{D})$. One can view the matrix $\mathfrak{i}(\mathcal{A})$ as an element of $M_{dn}(\mathcal{L})$. The reduced norm $\mathrm{Nrd}(\mathcal{A})$ is defined as  the determinant, over the commutative field $\mathcal{L}$, of the $dn \times dn$ matrix $\mathfrak{i}(\mathcal{A})$. One can show that this is an element of $\mathcal{F}^\times$. One can also show that the definition of the reduced norm is independent of the choice of the splitting field and the choice of the isomorphism $\mathcal{D} \otimes_\mathcal{F} \mathcal{L} \cong M_d(\mathcal{L})$.  See the description in Section 1.2.4 in Chapter III of Weibel's K-book \cite{MR3076731} for more details. If we let $f(t)$ denote the characteristic polynomial of the endomorphism $\mathcal{L}^{dn} \xrightarrow {\mathcal{A}} \mathcal{L}^{dn}$, of $\mathcal{L}$ vector-spaces, induced by the matrix $\mathfrak{i}(\mathcal{A})$, then one sees that $\mathrm{Nrd}(\mathcal{A})$ is also equal to the constant term of the polynomial $f(t)$. By abuse of notation, we will let $\mathrm{Nrd}$ also denote the following composition of maps $\mathcal{D}^\times \rightarrow K_1(\mathcal{D}) \xrightarrow {\mathrm{Nrd}} \mathcal{F}^\times $. \\

Concerning the properties of the restriction of the reduced norm of the division algebra to the subring $O_D$, see  Chapter 14 in Reiner's book \cite{MR0393100}. What we will need is the fact that $\mathrm{Nrd}(\pi_D)$ is a uniformizer in $O_F$. Let us denote $\mathrm{Nrd}(\pi_D)$ by $\pi_F$. \\

Now, we return to our setting. We have a natural inclusion of rings $R \hookrightarrow T$, where $R = M_m(O_{D_\Lambda})$ and $T=M_m(D_\Lambda)$. The set $\Gl_m(D_\Lambda) \cap M_m(O_{D_\Lambda})$, denoted by $S$ (say) is an Ore set inside $R$. We have a natural isomorphism $R_S \cong T$. Morita equivalence lets us obtain the natural isomorphism $$(R_S^*)^{ab} \overset{\det^{-1}}{\cong} K_1\left(M_m(D_\Lambda)\right) \underset{\mathrm{Morita}}{\cong} K_1(D_\Lambda) \overset{\det}{\cong} \left(D_\Lambda^*\right)^{ab}.$$

\begin{proposition} \label{image-det-max}
$\det(A)$ belongs to the image of the natural map
\begin{align*}
R \cap R_S^*  \rightarrow (R_S^*)^{ab}.
\end{align*}
of multiplicative monoids, where
\begin{align*}
R = M_m(O_{D_\Lambda}), \qquad R_S =M_m(D_\Lambda), \qquad A \in M_n(R) \cap \Gl_n(R_S), \qquad (R_S^*)^{ab} \cong (D_\Lambda^*)^{ab}.
\end{align*}
\end{proposition}

\begin{proof}
We proceed in several steps.

\underline{Step 1: The reduced norm of $A$ is integral}
\nopagebreak

We will show that $\mathrm{Nrd}(A)$ belongs to $O_{F_\Lambda}$. Let $d$ denote $\dim_{F_\Lambda}L_\Lambda$. We will fix an inclusion $\mathfrak{i}:D_\Lambda \hookrightarrow M_{d}(L_\Lambda)$ as in (\ref{inclusion-splitting-field}). Let $f(t)$ denote the characteristic polynomial (this is the reduced characteristic polynomial associated to the central simple algebra $M_{mn}(D_\Lambda)$ over $F_\Lambda$) of the endomorphism $L_\Lambda^{dmn} \xrightarrow {\mathfrak{i}(A)} L_\Lambda^{dmn}$, induced by the matrix $\mathfrak{i}(A)$. Note that $f(t)$ is an element of the polynomial ring $F_\Lambda[t]$. See Theorem 9.3 in Reiner's book \cite{MR0393100}. Note that $D_\Lambda$ is a vector space of dimension $d^2$ over $F_\Lambda$. Let $g(t)$ denote the characteristic polynomial of the endomorphism $F^{d^2mn}_{\Lambda} \xrightarrow {\alpha_A} F^{d^2mn}_{\Lambda}$, induced by the matrix $A$. Since the entries of the matrix $A$ lie in $O_{D_\Lambda}$, we have the following commutative diagram of $O_{F_\Lambda}$-modules:
\begin{align} \label{comm-diag-integral-endo}
\xymatrix{\left(O_{D_\Lambda} \right)^{mn} \ar[d]^{\cong} \ar[r]^{\alpha_A} & \left(O_{D_\Lambda} \right)^{mn} \ar[d]^{\cong} \\ O_{F_\Lambda}^{d^2mn} \ar[r]^{\alpha_A} & O_{F_\Lambda}^{d^2mn}}
\end{align}
The endomorphism $F^{d^2mn}_{\Lambda} \xrightarrow {A} F^{d^2mn}_{\Lambda}$ is induced by the $O_{F_\Lambda}$-module endomorphism given in (either row of) the commutative diagram in (\ref{comm-diag-integral-endo}). Thus, the polynomial $g(t)$  must belong to $O_{F_\Lambda}[t]$.  By Theorem 9.5 in Reiner's book \cite{MR0393100}, $f(t)$ divides $g(t)$, in $F_\Lambda[t]$. Note that the domain $O_{F_\Lambda}[t]$ is integrally closed. Note also that both $f(t)$ and $g(t)$ are monic polynomials. Since the coefficients of $g(t)$ are in $O_{F_\Lambda}$, so must the coefficients of $f(t)$. See Proposition 4.11 in Eisenbud's book \cite{eisenbud1995commutative}. Hence, $\mathrm{Nrd}(A)$ must lie in $O_{F_\Lambda}$. \\

\underline{Step 2: A non-commutative Weierstrass preparation theorem over $O_{D}[x]]$}
\nopagebreak

For the rest of the proof, we shall fix an isomorphism
\begin{align} \label{iso-powerseries}
O_{D_\Lambda} \cong O_D[[x]].
\end{align}
To each element $f =\sum_{n=0}^\infty a_n(f) x^n$ in $O_D[[x]]$, we can define a quantity called the reduced order of $f$, denoted $\mathrm{ord}(f)$, as follows:
\begin{align*}
\mathrm{ord}(f) = \min\left\{n \mid a_n \in O_D^\times\right\}.
\end{align*}
We set $\mathrm{ord}(f)$ to be $\infty$, if the set $\left\{n  \mid a_n \in O_D^\times\right\}$ is empty. Every non-zero element $f$ in $O_{D}[[x]]$ can be written as $\pi_D^{\mu} f_0$, where $f_0$ is some power series in $O_D[[x]]$ such that $\mathrm{ord}(f_0) < \infty$.

Just as in the commutative case, we have a Weierstrass preparation theorem over $O_D[[x]]$ too.
Firstly, let $f_1$ and $f_2$ be two elements in $O_D[[x]]$ such that $\mathrm{ord}(f_2) < \infty$. Then, there exists elements $a,b,r,s$ in $O_D[[x]]$ such that
\begin{align*}
f_1 = af_2 + r, \qquad f_1 = f_2b +s,
\end{align*}
and such that both $r$ and $s$ are polynomials whose degrees are less than $\mathrm{ord}(f_2)$.

Secondly, every element $f$ in $O_{D}[[x]]$ can be written as
\begin{align*}
f = \pi_D^{\mu_f} U_f J_f, \qquad \text{ where }U_f \in O_{D}[[x]]^\times, \ J_f \text{ is a monic polynomial under the isomorphism in (\ref{iso-powerseries})},
\end{align*}
and
\begin{align*}
f =  H_f V_f \pi_D^{\mu_f}, \qquad \text{ where } V_f \in O_{D}[[x]]^\times, \ H_f \text{ is a monic polynomial under the isomorphism in (\ref{iso-powerseries})}.
\end{align*}
These facts follow from the work of Venjakob \cite{MR1989649}. See Theorem 3.1 and Corollary 3.2 in \cite{MR1989649}. \\

\underline{Step 3: $O_{D_\Lambda}\left[\frac{1}{p}\right]$ is a non-commutative PID}
\nopagebreak

Theorem 14.3 in Reiner's book \cite{MR0393100} tells us that $u_0\pi_D^{d} = \pi_F$, for some  $u_0 \in O_D^\times$. As a result, there exists a positive integer $d'$ so that
\begin{align}\label{ppid}
u_1\pi_D^{d'} = p, \text{ where } u_1 \in O_D^\times.
\end{align}
Since $p$ is invertible in $O_{D_\Lambda}\left[\frac{1}{p}\right]$, so is $\pi_D$. This observation along with the Weierstrass preparation theorem allows us to conclude that $O_{D_\Lambda}\left[\frac{1}{p}\right]$ is a non-commutative PID. To see this: for each non-zero left (right) ideal $I$ in $O_{D_\Lambda}\left[\frac{1}{p}\right]$, choose a monic polynomial $f$ in $I$ with least reduced order. A standard application of the Weierstrass preparation theorem, just as in the commutative case, will show us that this element $f$ is a generator for the left (right) ideal $I$.

We would like to make two further useful observations:
\begin{enumerate}
\item If $J$ is a monic polynomial of degree $r$, then $\mathrm{Nrd}(J)$ is a monic polynomial of degree $rd$.
\item Every unit in the ring $O_{D_\Lambda}\left[\frac{1}{p}\right]$ is of the form $\pi_D^{r}\beta$, for some integer $r$ and some $\beta$ in $O_{D_\Lambda}^\times$.
\end{enumerate}

\underline{Step 4: Diagonal reduction over $O_{D_\Lambda}\left[\frac{1}{p}\right]$:}
\nopagebreak

Proposition \ref{diagonal-reduction} tells us that the matrix $A$ admits a diagonal reduction via elementary operations in $M_{mn}\left(O_{D_\Lambda}\left[\frac{1}{p}\right]\right)$. So, there exists a diagonal matrix $B$ in $M_{mn}\left(O_{D_\Lambda}\left[\frac{1}{p}\right]\right)$ and invertible matrices $U$ and $V$ (obtained as products of elementary matrices, permutation matrices and scalar matrices in $\Gl_{mn}\left(O_{D_\Lambda}\left[\frac{1}{p}\right]\right)$) so that $A = UBV$. This allows us to obtain following equality in $K_1(D_\Lambda)$:
\begin{align*}
\det(U) = \pi_D^{r_U} \beta_U , \qquad \det(V) = \pi_D^{r_V} \beta_V,
\end{align*}
where $r_U$ and $r_V$ are integers while $\beta_U$ and $\beta_V$ are elements of $O_{D_\Lambda}^\times$. Since $B$ is a diagonal matrix, by multiplying all the elements in the main diagonal of $B$, we obtain the following equality in $K_1(D_\Lambda)$:
\begin{align*}
\det(B) = \pi^{r_B}_D \beta_B J_B,
\end{align*}
where $r_B$ is an integer, $\beta_B$ is an element of $O_{D_\Lambda}^\times$ and $J_B$ is a monic polynomial in $O_{D_\Lambda}$ (under the isomorphism $O_{D_\Lambda} \cong O_D[[x]]$ given in (\ref{iso-powerseries})).
Set
\begin{align*}
J_A:=J_B \in O_D[x], \qquad r_A:= r_U + r_B + r_V \in \Z, \qquad \beta_A := \beta_U \beta_B \beta_V \in O_{D_\Lambda}^\times.
\end{align*}

\underline{Step 5: Completing the proof:}
\nopagebreak

Set \begin{align*}
C := \left[\begin{array} {cccc} \beta_A J_A\pi_D^{r_A}  & 0 & \ldots & 0 \\ 0 & 1 & \ldots & 0 \\ \vdots & & \ddots &  \\ 0 & 0 & \ldots & 1 \end{array}\right]\in M_m\left(O_{D_\Lambda}\left[\frac{1}{p}\right]\right).
\end{align*}
Since $K_1(D_\Lambda)$ is an abelian group, we have the following equality in $K_1(D_\Lambda)$:
\begin{align*}
\det(A) & = \det(U) \det(B) \det(V) \\ & =\pi_D^{r_U} \beta_U \cdot \pi^{r_B}_D \beta_B J_B \cdot \pi_D^{r_V} \beta_V \\ & = \beta_U \beta_B \beta_V \cdot J_B \cdot \pi_D^{r_U + r_B + r_V} \\ &= \beta_A J_A\pi_D^{r_A} = \det(C).
\end{align*}
Now, to complete the proof of the proposition, we will show that $r_A$ is non-negative. This would tell us that $C$ is a matrix in $M_m\left(O_{D_\Lambda}\right)$ and that $C$ is a representative for $\det(A)$ in $M_m\left(O_{D_\Lambda}\right)$.

Computing reduced norms, we obtain the following equality in $F_\Lambda$:
\begin{align*}
 \mathrm{Nrd}(A) &= \mathrm{Nrd}(\det(A)) = \mathrm{Nrd}(\beta_A J_A\pi_D^{r_A}) \\
\implies  \mathrm{Nrd}(A) &= \mathrm{Nrd}(\beta_A)  \cdot \mathrm{Nrd}(J_A) \cdot \pi_F^{dr_A}.
\end{align*}
We have shown that $\mathrm{Nrd}(A)$ is an element of $O_{F_\Lambda}$. So, $ \mathrm{Nrd}{\beta_A}  \cdot \mathrm{Nrd}(J_A) \cdot \pi_F^{dr_A}$ must belong to the unique factorization domain $O_{F_\Lambda}$ as well. Since $\beta_A$ is a unit in the ring $O_{D_\Lambda}$, the element $\mathrm{Nrd}(\beta_A)$ is a unit in $O_{F_\Lambda}$. This follows from Theorem 10.1 in Reiner's book \cite{MR0393100} and the fact the reduced norm is a group homomorphism. The irreducible $\pi_F$ cannot divide the monic polynomial $\mathrm{Nrd}(J_A)$. As a result, $dr_A$ must be non-negative and hence, so must the integer $r_A$. This completes the proof of the proposition.
\end{proof}

\section{The non-commutative Iwasawa main conjecture and Proof of Theorem \ref{integrality-determinant}} \label{section-non-commutative}

We will readily borrow the terminologies used in Weibel's K-book \cite{MR3076731}, the work of Fukaya-Kato \cite{MR2276851} and Section 2 of Kakde's work \cite{MR3091976}  to describe various objects appearing in the non-commutative Iwasawa main conjecture. The ``canonical'' Ore sets $\mathfrak{S}$ and $\mathfrak{S}^*$, that come into play, are given below:

\begin{align*}
\mathfrak{S} &= \left\{ s \in \Lambda[G],\text{ such that }  \frac{\Lambda[G]}{\Lambda[G]s} \text{ is a finitely generated $\Z_p$-module} \right\}, \quad
\mathfrak{S}^* &= \bigcup \limits_{n \geq 0} p^n \mathfrak{S}, \quad S^* = \Lambda \setminus \{0\}.
\end{align*}

In this article, we will only consider the localization $\Lambda[G]_{\mathfrak{S}^*}$. The set $\mathfrak{S}^*$ is a multiplicatively closed set, consisting of non-zero divisors in $\Lambda[G]$. Since the group $G$ is finite, we have the isomorphisms
\begin{align*}
\Lambda[G]_{\mathfrak{S}^*} \cong \Lambda[G]_{S^*}\cong Q_\Lambda[G].
\end{align*}
It will be advantageous to work with $S^*$ since all the elements of $S^*$ are central. To formulate the main conjecture, we will have to consider the connecting homomorphism (obtained from the localization sequence in $K$-theory):
\begin{align*}
\partial: K_1\left(Q_\Lambda[G] \right) \rightarrow  K_0\left(\Lambda[G], Q_\Lambda[G] \right).
\end{align*}

\subsection{Interpolation properties of the $p$-adic $L$-functions} \label{subsec:inter_prop} \mbox{}

To describe the interpolation properties of $p$-adic $L$-functions, we will follow the illustrations provided in the works of Johnston-Nickel \cite[Section 4.3]{johnstonhybrid} and Ritter-Weiss \cite[Section 4]{MR2114937}.

Let $Z(Q_\Lambda[G])$ denote the center of $Q_\Lambda[G]$. Corresponding to each Weddernburn component $M_{m_i}(D_i)$ of $\Q_p[G]$ appearing in equation (\ref{AW-groupring}), we let $F_i$ denote the center $Z(D_i)$ and $n_i$ denote $\mathrm{dim}_{F_i}(D_i)$. Note that $F_i$ is a finite extension of $\Q_p$. Let $O_{F_i}$ denote the ring of integers of $F_i$. Note also that every simple (left) module of $M_{m_i}(D_i)$ is isomorphic to the simple  module $D_i^{m_i}$ with the natural (left) action of the matrix ring $M_{m_i}(D_i)$. As a result, each Weddernburn component $M_{m_i}(D_i)$ corresponds uniquely to an irreducible (totally even) Artin representation $\rho_i: \Gal{L}{K} \rightarrow \Gl_{n_im_i}(F_i)$, given as follows:
\begin{align*}
\rho_i: \underbrace{\Gal{L}{K}}_{G} \hookrightarrow \Q_p[G]^\times \twoheadrightarrow \GL_{m_i}(D_i) \cong \Aut \left(D_i^{m_i}\right) \cong \Gl_{m_in_i}(F_i).
\end{align*}
To describe the local Euler factors at primes $\nu \in \Sigma_0$, we will follow the illustration provided in work of Greenberg-Vatsal \cite[Proposition 2.4]{greenberg2000iwasawa}. Suppose $F$ denotes a finite extension of $\Q_p$. Suppose $\rho:G \rightarrow \GL_n(F)$ denotes an Artin representation. Let $V$ and $V^*:=\Hom_F\left(V,F(1)\right)$ denote the $F[G]$-modules corresponding to $\rho$ and its Tate dual $\rho^*$ respectively. For each $\nu$ in $\Sigma_0$, we consider $P_{\nu,\rho}(x):=\det\left(1-x \rho^*\mid_{V^*_{I_\nu}}(\Frob^{-1}_\nu)\right)$ in $F[x]$. Here, $V^*_{I_\nu}$ denotes the maximal quotient of $V^*$ on which the inertia group $I_\nu$ acts trivially.  Here, $I_\nu$ denotes the inertia subgroup inside $\Gal{\overline{K}_\nu}{K_\nu}$ and $\Frob_\nu$ is the Frobenius element in $\frac{\Gal{\overline{K}_\nu}{K_\nu}}{I_\nu}$ . Let $\Gamma_\nu$ denote a decomposition group corresponding to $\nu$ inside $\Gamma$. We can naturally view $\gamma_\nu$, the Frobenius automorphism at $\nu$ of $\Gamma_\nu$, as an element of $\Gamma$ via the inclusion $\Gamma_\nu \subset \Gamma$. We will let $f_{\nu,\rho}$ denote $P_{\nu,\rho}(\gamma_\nu)$, viewed as an element of $\mathrm{Frac}(O_F[[\Gamma]])$.

Suppose first that the character $\chi$ is totally even. For each $i$, note that $\chi \rho_i$ then is a totally even Artin representation of $K$ of ``type $S$''. Greenberg \cite[Section 2]{MR692344} has constructed a primitive $p$-adic $L$-function $L_{p,\chi \rho_i}$  as an element of $\mathrm{Frac}\left(O_{F_i}[[\Gamma]]\right)$. In this case, we let
\begin{align*}
\Phi^{\Sigma_0}_{\chi} := \bigg(L_{p,\chi \rho_i} \prod_{\nu \in \Sigma_0} f_{\nu,\chi\rho_i}\bigg)_i \qquad \text{in } Z(Q_\Lambda[G]).
\end{align*}

Suppose now that the character $\chi$ is totally odd. For each $i$, note that $\chi^{-1}\omega\rho_i^{-1}$ then is a totally even Artin representation of $K$ of ``type $S$''. Let $\iota: \Z_p[[\Gamma]] \rightarrow \Z_p[[\Gamma]]$ denote the $\Z_p$-linear ring homomorphism induced by sending $\gamma \rightarrow \gamma^{-1}<\gamma>$, for each $\gamma$~in~$\Gamma$. Here, we obtain the element $<\gamma>$ via the canonical injection $<>:\Gamma \hookrightarrow \Gal{\Q_\mathrm{cyc}}{\Q} \xrightarrow {\cong} 1+p\Z_p$. In this case, we let
\begin{align*}
\Phi^{\Sigma_0}_{\chi} := \bigg(\iota\left(L_{p,\chi^{-1}\omega \rho_i^{-1}}\right) \prod_{\nu \in \Sigma_0} f_{\nu,\chi\rho_i} \bigg)_i \qquad \text{in } Z(Q_\Lambda[G]).
\end{align*}

\begin{comment}
\end{comment}

\begin{conjecture}[Interpolation property for the $p$-adic $L$-function]\label{conj:int_prop}\mbox{}

There exists a unique element $\xi$ in $K_1\left(Q_\Lambda[G]\right)$ such that $\mathrm{Nrd}(\xi)=\Phi^{\Sigma_0}_{\chi}$ in $Z(Q_\Lambda[G])$.
\end{conjecture}

\begin{remark}
The fact that the primitive $p$-adic $L$-function for a totally odd Artin representation $\rho$ is related to the primitive $p$-adic $L$-function for the corresponding totally even Artin representation $\rho^{-1}\omega$ is also mirrored on the algebraic side. In the commutative setting, this is the ``reflection principle''. See works of Greenberg \cite[Section 2]{MR0444614} and \cite[Theorem 2]{greenberg1989iwasawa}. The analog of the reflection principle in the non-commutative setting is discussed, from both the algebraic and analytic perspective, in the work of Fukaya-Kato \cite[Section 4.4]{MR2276851}.
\end{remark}

\subsection{The non commutative Iwasawa main conjecture}

\mbox{}

On the algebraic side of Iwasawa theory, one works with the relative $K_0$-group $K_0\left(\Lambda[G], Q_\Lambda[G]\right)$. One can give two different descriptions of this relative $K_0$-group. The first description involves the exact category $\mathbf{H}_{1,S^*}$. The category $\mathbf{H}_{1,S^*}$ is a subcategory of the category of finitely generated left $\Lambda[G]$-modules. The objects of $\mathbf{H}_{1,S^*}$ are $\Lambda[G]$-modules that are $S^*$-torsion and that have projective dimension less than or equal to one. To define the relative $K_0$-group $K_0\left(\Lambda[G], Q_\Lambda[G]\right)$, we refer the reader to Definition 2.10 and Exercise 7.11 in Chapter II of Weibel's K-book \cite{MR3076731}. We will need to consider tuples $(P_1,\alpha,P_2)$, where
\begin{itemize}
\item $P_1$ and $P_2$ are projective $\Lambda[G]$-modules, and
\item the map $\alpha:Q_\Lambda[G]\otimes_{\Lambda[G]} P_1 \rightarrow Q_\Lambda[G] \otimes_{\Lambda[G]} P_2$ is an isomorphism of $Q_\Lambda[G]$-modules.
\end{itemize}
The relative $K_0$-group $K_0\left(\Lambda[G], Q_\Lambda[G]\right)$ is defined to be the quotient of the free abelian group generated by such tuples $(P_1,\alpha,P_2)$ subject to the following two relations:
\begin{enumerate}[(i)]
\item $[(P_1,\alpha_1,Q_1)] + [(P_3,\alpha_2,Q_3)]= [(P_2,\alpha_2 \circ \alpha_1,Q_2)]$, whenever we have two exact sequences of projective $\Lambda[G]$-modules
\begin{align*}
0 \rightarrow P_1 \rightarrow P_2 \rightarrow P_3 \rightarrow 0 ,\qquad 0 \rightarrow Q_1 \rightarrow Q_2 \rightarrow Q_3 \rightarrow 0,
\end{align*}
along with an induced commutative diagram of $Q_\Lambda[G]$-modules with exact rows:
\begin{align*}
\xymatrix{
0 \ar[r]& Q_\Lambda[G]\otimes_{\Lambda[G]} P_1 \ar[d]^{\alpha_1}_\cong \ar[r]& Q_\Lambda[G]\otimes_{\Lambda[G]} P_2 \ar[r]\ar[d]^{\alpha_2}_\cong & Q_\Lambda[G]\otimes_{\Lambda[G]} P_3 \ar[d]^{\alpha_3}_\cong \ar[r]&0 \\
0 \ar[r]& Q_\Lambda[G]\otimes_{\Lambda[G]} Q_1 \ar[r]& Q_\Lambda[G]\otimes_{\Lambda[G]} Q_2 \ar[r]& Q_\Lambda[G]\otimes_{\Lambda[G]} Q_3 \ar[r]&0
}
\end{align*}
\item $[(P_1,\alpha_{21},P_2)] + [(P_2,\alpha_{32},P_3)]= [(P_1,\alpha_{32} \circ \alpha_{21},P_3)]$.
\end{enumerate}

One can give a second description of this relative $K_0$-group involving the Waldhausen category $\mathbf{Ch}_{S^*}^\flat\big(\mathbf{P}(\Lambda[G])\big)$. This is the category of bounded chain complexes of finitely generated projective $\Lambda[G]$-modules whose cohomologies are $S^*$-torsion. In \cite{MR2276851}, Fukaya and Kato use the second description of this relative $K_0$-group to formulate the non-commutative Iwasawa main conjecture.  Fukaya and Kato construct an element  of this category $\mathbf{Ch}_{S^*}^\flat\big(\mathbf{P}(\Lambda[G])\big)$, whose cohomology is closely related to $\X$. Fukaya and Kato  label this chain complex $\mathrm{SC}\left(U,T,T^0\right)$. We will follow their notations to describe this chain complex.

If $\chi$ is totally even, we have
\begin{align*}
U = \Sigma_0, \quad T = \Z_p(\chi^{-1}\chi_p), \qquad T^0 = 0,
\end{align*}
and the cohomology of the chain complex $\mathrm{SC}\left(U,T,T^0\right)$ is given below:
\begin{align} \label{eq:desc_complex_even}
H^i\bigg(\mathrm{SC}\left(U,T,T^0\right)\bigg) = \left\{ \begin{array}{cc} \Z_p, & \text{if $i =3$ and $\chi$ is trivial,} \\ \X , & \text{if $i=2$,} \\ 0, & \text{otherwise}. \end{array} \right.
\end{align}
Here, $\chi_p:G_\Sigma \rightarrow \Z_p^\times$ denotes the $p$-adic cyclotomic character given by the action of $G_\Sigma$ on the $p$-power roots of unity $\mu_{p^\infty}$.

If $\chi$ is totally odd, we have
\begin{align*}
U = \Sigma_0, \quad T = \Z_p(\chi^{-1}\chi_p), \qquad T^0 =  \Z_p(\chi^{-1}\chi_p).
\end{align*}
As for the cohomology of the chain complex when $\chi$ is totally odd, we have $H^i\bigg(\mathrm{SC}\left(U,T,T^0\right)\bigg)=0$, if $i \neq 2$. We also have the following exact sequence:
\begin{align} \label{complexes-iwasawa-module}
0 \rightarrow  \bigoplus \limits_{\omega \in \Sigma_p(L_\infty)} H^1\left(\Gamma_\omega, H^0\left(I_\omega, \frac{\Q_p(\chi)}{\Z_p(\chi)}\right)\right)^\vee \rightarrow & \X \rightarrow H^2\bigg(\mathrm{SC}\left(U,T,T^0\right)\bigg) \rightarrow  \\ & \notag \rightarrow \bigoplus \limits_{\omega \in \Sigma_p(L_\infty)}  H^0\left(G_\omega, \frac{\Q_p(\chi)}{\Z_p(\chi)}\right)^\vee \rightarrow 0.
\end{align}

\begin{remark}
A word of caution about the terminology in \cite{MR2276851}: the module labeled $\mathcal{X}(T,T^0)$ in Fukaya and Kato's work \cite{MR2276851} is the Pontryagin dual of the ``strict'' Selmer group.
\end{remark}

\begin{remark}\label{odd-remark}
There is a nice illustration on how to compute the cohomology of the chain complex $\mathrm{SC}\left(U,T,T^0\right)$ in Examples 4.5.1 and 4.5.2 of Fukaya and Kato's work \cite{MR2276851}. We have mainly followed those illustrations. See Section 2.3 in Kakde's work for the description of the cohomology of the chain complex $\mathrm{SC}(U,T,T^0)$ when the character $\chi$ is totally even.

When $\chi$ is totally odd, we will need to use the description of $\mathrm{SC}(U,T,T^0)$ given in equation (4.1) in Section 4.1.2 of Fukaya and Kato's work \cite{MR2276851}. The illustration given in the proof of Proposition 4.2.35 in \cite{MR2276851} is  helpful for this computation. The fact that $H^1\bigg(\mathrm{SC}\left(U,T,T^0\right)\bigg)=0$ crucially relies on the observation that the global-local map defining the non-primitive ``strict'' Selmer group is surjective.
\end{remark}

For the definition of $K_0\left(\mathbf{Ch}_{S^*}^\flat\big(\mathbf{P}(\Lambda[G])\big)\right)$, we refer the reader to Definition 9.1.2 in Chapter II of Weibel's K-book \cite{MR3076731}. For our purposes, we will simply keep in mind that $K_0\left(\mathbf{Ch}_{S^*}^\flat\big(\mathbf{P}(\Lambda[G])\big)\right)$ is a certain quotient of the free abelian group generated by the objects of $\mathbf{Ch}_{S^*}^\flat\big(\mathbf{P}(\Lambda[G])\big)$. Using the second description of the relative $K_0$-group involving $\mathbf{Ch}_{S^*}^\flat\big(\mathbf{P}(\Lambda[G])\big)$ allows us to consider the element $\left[\mathrm{SC}\left(U,T,T^0\right)\right]$.

We will follow the formulation of the non-commutative Iwasawa main conjecture given in work of Johnston-Nickel \cite[Conjecture 4.4]{johnstonhybrid}. See also work of Fukaya-Kato \cite{MR2276851} and Ritter-Weiss \cite{MR2114937}.

\begin{conjecture} \label{conj:noncomm}
Conjecture \ref{conj:int_prop} holds. Furthermore, we have the following equality in $K_0\left(\Lambda[G], Q_\Lambda[G]\right)$:
\begin{align}
 \partial(\xi) \stackrel{?}{=} \left[\mathrm{SC}\left(U,T,T^0\right)\right].
\end{align}
\end{conjecture}

Let $\mathbf{H}_{S^*}$ denote the exact subcategory of the category of $\Lambda[G]$-modules, whose objects are finitely generated $\Lambda[G]$-modules  that are $S^*$-torsion and that have finite projective dimension. It turns out that we have the following natural isomorphisms:
\begin{align*}
K_0(\Lambda[G],Q_\Lambda[G]) & \cong K_0\left(\mathbf{H}_{1,S^*}\right) \\ & \cong  K_0\left(\mathbf{H}_{S^*}\right), \qquad \text{Corollary 7.7.3 to the Resolution Theorem 7.6 in Chapter II of \cite{MR3076731}}, \\ & \cong K_0\left(\mathbf{Ch}_{S^*}^\flat\big(\mathbf{P}(\Lambda[G])\big)\right), \qquad \text{Exercise 9.13 in Chapter II of \cite{MR3076731}}.
\end{align*}

From the perspective of homological algebra, one difficulty with using the first description of the relative $K_0$-group is that the Iwasawa algbera $\Lambda[G]$ may have infinite global dimension. As a result, one really does need to use this workaround to work with an element in the relative $K_0$-group. Suppose, as when Theorems \ref{evenprojectivity} and \ref{oddprojectivity}  indicate, for the rest of this section that the $\Lambda[G]$-module $\X$ has a free resolution of length one. That is, we have the following short exact sequence of $\Lambda[G]$-modules:
\begin{align} \label{freeresol}
0 \rightarrow \Lambda[G]^n \xrightarrow {A_\X} \Lambda[G]^n \rightarrow \X \rightarrow 0.
\end{align}
We have the following equality in the $K_0(\Lambda[G],Q_\Lambda[G])$:
\begin{align} \label{equality-complex-module}
\left[\mathrm{SC}\left(U,T,T^0\right)\right] = \left[\left(\Lambda[G],A_\X, \Lambda[G]\right)\right].
\end{align}
When $\chi$ is totally even and non-trivial, this equality follows from equation (\ref{eq:desc_complex_even}).

Let us now see why equality holds in equation (\ref{equality-complex-module}) when the character $\chi$ is totally odd. Note that $\mathbf{H}_{S^*}$ is closed under kernels of surjections inside the abelian category of finitely generated (left) $\Lambda[G]$-modules. By Theorem 9.2.2 in Chapter II of Weibel's K-book \cite{MR3076731}, we have the natural isomorphism $K_0\left(\mathbf{H}_{S^*}\right) \cong K_0\left(\mathbf{Ch}^\flat\left(\mathbf{H}_{S^*}\right)\right)$. Here, $\mathbf{Ch}^\flat\left(\mathbf{H}_{S^*}\right)$ is the category of bounded chain complexes in $\mathbf{H}_{S^*}$. By the same theorem, the equality in equation (\ref{equality-complex-module}) would follow if one can show that the Euler characteristic of the chain complex $\X \rightarrow H^2\bigg(\mathrm{SC}\left(U,T,T^0\right)\bigg)$, obtained from equation (\ref{complexes-iwasawa-module}), equals zero in $K_0\left(\mathbf{Ch}^\flat\left(\mathbf{H}_{S^*}\right)\right)$.

Theorem \ref{oddprojectivity} tells us that $\X$ has a free resolution of length one under one of the following conditions:
\begin{enumerate}[label=(\Roman*), ref=\Roman*]
\item\label{condn-remark-I} $H^0\left(G_\omega,\mathfrak{D}(\chi)\right)=0$.
\item\label{condn-remark-II} $\omega$ is tamely ramified in the extension $L_\infty / K_\infty$.
\end{enumerate}
One can compare the modules on either side of the exact sequence (\ref{complexes-iwasawa-module}) using the observations in Section \ref{first-odd-case}. If condition \ref{condn-remark-I} holds for the prime $\omega$ in $\Sigma_p(L_\infty)$, then
\begin{align*}
H^1\left(\Gamma_\omega, H^0\left(I_\omega, \frac{\Q_p(\chi)}{\Z_p(\chi)}\right)\right)^\vee \cong H^0\left(G_\omega, \frac{\Q_p(\chi)}{\Z_p(\chi)}\right)^\vee  =0.
\end{align*}
If condition \ref{condn-remark-I} does not hold and condition \ref{condn-remark-II} holds  for the prime $\omega$ in $\Sigma_p(L_\infty)$, then
\begin{align*}
H^1\left(\Gamma_\omega, H^0\left(I_\omega, \frac{\Q_p(\chi)}{\Z_p(\chi)}\right)\right)^\vee \cong H^0\left(G_\omega, \frac{\Q_p(\chi)}{\Z_p(\chi)}\right)^\vee  \cong \Z_p.
\end{align*}
The equality in equation (\ref{equality-complex-module}) follows from these observations and equation (\ref{complexes-iwasawa-module}).

\subsection{Proof of Theorem \ref{integrality-determinant}} \mbox{}

We let  \begin{align} \label{maximal-order-definition}
M_{\Lambda[G]} := \prod M_{m_i}\left(O_{D_i} \otimes_{\Z_p} \Lambda\right).
\end{align}
Proposition 10.5 in Reiner's book \cite{MR0393100} and Proposition \ref{establish-max-order} tell us that $M_{\Lambda[G]}$ is a maximal $\Lambda$-order containing $\Lambda[G]$ inside $Q_\Lambda[G]$.

\intdeterminant*

\begin{proof}
Consider the following resolution of the $\Lambda[G]$-module $\X$:
\begin{align}
0 \rightarrow \Lambda[G]^n \xrightarrow {A_\X} \Lambda[G]^n \rightarrow \X \rightarrow 0.
\end{align}
Here, $A_\X$ is a matrix in $M_n(\Lambda[G]) \cap \Gl_n(Q_\Lambda[G])$. Note that, since Conjecture \ref{conj:noncomm} is assumed to hold, we have the following equality in $K_0\left(\Lambda[G], Q_\Lambda[G] \right)$:
\begin{align}
\partial(\xi) = \partial\left(A_\X\right) = \left(\Lambda[G]^n,A_\X,\Lambda[G]^n\right).
\end{align}
The localization exact sequence in $K$-theory gives us the following exact sequence
\begin{align*}
K_1(\Lambda[G]) \rightarrow K_1(Q_\Lambda[G]) \xrightarrow {\partial} K_0\left(\Lambda[G], Q_\Lambda[G] \right).
\end{align*}
Since the ring $\Lambda[G]$  is also a semi-local ring, the Dieudonn\'e determinant also provides us an isomorphism $K_1(\Lambda[G]) \cong (\Lambda[G]^*)^{ab}$. We have $$\det(\xi) = \det\left(A_\X\right) \det(B) \in (Q_\Lambda[G]^*)^{ab},$$
where $B$ in a matrix in $\Gl_\infty(\Lambda[G])$. The isomorphism $K_1(\Lambda[G]) \cong (\Lambda[G]^*)^{ab}$ allows us to find a representative for $\det(B)$ in $\Lambda[G]$. To prove the theorem, it now suffices to show that $\det\left(A_\X\right)$ belongs to the image of the natural map of multiplicative monoids:
\begin{align*}
M_{\Lambda[G]} \cap Q_\Lambda[G]^* \rightarrow K_1\left(Q_\Lambda[G]\right)
\end{align*}
The Artin-Weddernburn theorem and equation (\ref{comm-diag-artin}) gives us the following isomorphism:
\begin{align} \label{different-division-algebras}
Q_\Lambda[G] \cong \prod_i M_{m_i}(D_i\otimes_{\Lambda} Q_\Lambda)
\end{align}
Let $\sigma_i$ denote the projection onto the $i$-th factor. The description of the maximal order, given in (\ref{maximal-order-definition}), now allows us to work with each factor in the product decomposition of (\ref{different-division-algebras}). It now suffices to show that $\det\left(\sigma_i\left(A_\X\right)\right)$ belongs to the image of the natural map of multiplicative monoids:
\begin{align*}
M_{m_i}\left(O_{D_i} \otimes_{\Z_p} \Lambda\right) \bigcap \Gl_{m_i}(D_i\otimes_{\Lambda} Q_\Lambda) \rightarrow K_1\left(M_{m_i}(D_i\otimes_{\Lambda} Q_\Lambda)\right) \cong (D_i\otimes_{\Lambda} Q_\Lambda)^{ab}
\end{align*}
This last statement, and hence the theorem, follows from Proposition \ref{image-det-max}.
\end{proof}

\section{Cohomological criterion}

In this section, we recall the cohomological criterion developed by Ralph Greenberg in the AMS memoir \cite{greenberg2011iwasawa} on Iwasawa theory, projective modules and modular representations.

A theorem of Iwasawa \cite{MR0349627} (see also Proposition 1 in Greenberg's work on $p$-adic Artin $L$-functions \cite{greenberg2014p}) asserts that $\X$ is a torsion module over $\Lambda[G]$.  When $\chi$ is totally even, Proposition 6.10 (along with the validity of the Weak Leopoldt conjecture) in Greenberg's work on the structure of Galois cohomology groups \cite{MR2290593} asserts that $\X$ has no non-zero finite $\Lambda$-submodules. See Theorem 10.3.25 in the book by Neukirch, Schmidt and Winberg \cite{neukirch2008cohomology} as to why the weak Leopoldt conjecture is valid in this setting. When $\chi$ is totally odd, the discussion in Section 4.4 of Greenberg's recent work on the structure of Selmer groups \cite{greenberg2014pseudonull} asserts that $\X$ has no non-zero finite $\Lambda$-submodules. These results allow us to apply the cohomological criterion developed by Greenberg (Proposition 2.4.1 in \cite{greenberg2011iwasawa}).

\begin{proposition}[Proposition 2.4.1 in \cite{greenberg2011iwasawa}] \label{greenberg-freeness}
The $\Lambda[G]$-module $\X$ has a free resolution of length one if and only if there exists a positive integer $m$ such that \begin{align}\label{vanishing-criterion}
H^{m}\left(P, \Sel_{\mathfrak{D}(\chi)}(L_\infty)\right)=0, \qquad H^{m+1}\left(P, \Sel_{\mathfrak{D}(\chi)}(L_\infty)\right)=0
\end{align}
for every subgroup $P$ of $P_G$. Here, $P_G$ is some $p$-Sylow subgroup of $G$.
\end{proposition}

\begin{remark}
Though Proposition 2.4.1 in \cite{greenberg2011iwasawa} requires us to verify the vanishing criterion (given in (\ref{vanishing-criterion})) for all subgroups of $G$, it suffices to restrict ourselves to subgroups of a $p$-sylow subgroup $P_G$. This is because every element of the discrete module $\Sel_{\mathfrak{D}(\chi)}(L_\infty)$ is killed by a power of $p$. Furthermore, Proposition 2.4.1 in \cite{greenberg2011iwasawa} establishes that $\X$ has a free resolution of length one if and only if for every subgroup $P$ of $P_G$
\begin{align}\label{greenberg-coho-vanishing}
H^{i}\left(P, \Sel_{\mathfrak{D}(\chi)}(L_\infty)\right)=0, \qquad \ \forall \  i \geq 1.
\end{align}
Theorem 4.2.3 in Hida's book \cite{hida2000modular} allows us to deduce that the validity of (\ref{vanishing-criterion}) implies the validity of equation (\ref{greenberg-coho-vanishing}).
\end{remark}

\begin{remark} \label{nmidG}
When $p$ does not divide $|G|$, the cohomology groups appearing in (\ref{vanishing-criterion}) vanish. As the cohomological criterion in Proposition \ref{greenberg-freeness} illustrates, in this case, the $\Lambda[G]$-module $\X$ has a free resolution of length one.
\end{remark}

\subsection{Global cohomology groups and Proof of Theorem \ref{evenprojectivity}}

In this section, we want to prove the following theorem stated in the introduction.

\evprojectivity*

Before proving the theorem, let us  introduce some notations.  Let $P_G$ denote a $p$-Sylow subgroup of $G$. Let $P$ be a subgroup of $P_G$. By Galois theory, we can identify $P$ with a  Galois group $\Gal{L_\infty}{F_\infty}$, for some field $F_\infty$ such that $K_\infty \subset L_\infty^{P_G} \subset F_\infty \subset L_\infty$. We have
 \begin{align*}
P \cong \Gal{L_\infty}{F_\infty}.
\end{align*}

\begin{lemma} \label{diff-maps-iso}
The differential maps in the Hochschild-Serre spectral sequence
 \begin{align*}
H^i\bigg(\Gal{L_\infty}{F_\infty}, \ H^j\big(\Gal{K_\Sigma}{L_\infty},\mathfrak{D}(\chi)\big)\bigg) \implies H^{i+j}\big(\Gal{K_\Sigma}{F_\infty},\mathfrak{D}(\chi)\big),
\end{align*}
yield the following isomorphism, for each $i \geq 1$:
 \begin{align*}
H^i\bigg(\Gal{L_\infty}{F_\infty}, \ H^1\big(\Gal{K_\Sigma}{L_\infty},\mathfrak{D}(\chi)\big)\bigg) \cong H^{i+2}\bigg(\Gal{L_\infty}{F_\infty}, \ H^0\big(\Gal{K_\Sigma}{L_\infty},\mathfrak{D}(\chi)\big)\bigg).
\end{align*}
\end{lemma}
\begin{proof}
The $p$-cohomological dimensions of $\Gal{K_\Sigma}{L_\infty}$ and $\Gal{K_\Sigma}{F_\infty}$ are less than or equal to  $2$. The validity of the Weak Leopoldt conjecture (Theorem 10.3.25 in the book by Neukirch, Schmidt and Winberg \cite{neukirch2008cohomology}) allows us to conclude that
 \begin{align*}
H^2\big(\Gal{K_\Sigma}{L_\infty},\mathfrak{D}(\chi)\big) =0, \qquad H^2\big(\Gal{K_\Sigma}{F_\infty},\mathfrak{D}(\chi)\big) =0.
\end{align*}

These observations combine to give us the following equalities:
\begin{align*}
H^i\bigg(\Gal{L_\infty}{F_\infty}, \ H^j\big(\Gal{K_\Sigma}{L_\infty},\mathfrak{D}(\chi)\big)\bigg) =0, \qquad \forall j \geq 2.  \\
H^j\big(\Gal{K_\Sigma}{F_\infty},\mathfrak{D}(\chi)\big) =0, \qquad \forall j \geq 2.
\end{align*}
This completes the proof of the lemma.
\end{proof}

Suppose that the character $\chi$ is totally odd. In this case, we have  $H^0\big(\Gal{K_\Sigma}{L_\infty},\mathfrak{D}(\chi)\big)=0$. This observation uses the fact that  $L_\infty$ is a totally real field. As an immediate consequence of Lemma \ref{diff-maps-iso}, we obtain the following result.

\begin{lemma} \label{odd-global}
Suppose that the character $\chi$ is totally odd. We have the following equality:
\begin{align}
H^i\bigg(\Gal{L_\infty}{F_\infty}, \ H^1\big(\Gal{K_\Sigma}{L_\infty},\mathfrak{D}(\chi)\big)\bigg) = 0.
\end{align}
\end{lemma}

We will use the following simple observation frequently in this paper.
\begin{observation} \label{non-trivial-residual-obs}
The finite character $\chi : G_\Sigma \rightarrow \Z_p^\times$ is non-trivial if and only if the residual character \mbox{$\overline{\chi}~:~G_\Sigma \xrightarrow{\chi}~\Z_p^\times~\rightarrow~\mathbb{F}_p^\times$} associated to it is non-trivial.
\end{observation}

We now proceed to the proof of Theorem \ref{evenprojectivity}.

\begin{proof}[Proof of Theorem \ref{evenprojectivity}]
We are working under the assumption that the character $\chi$ is totally even. In this case, note that
$ \Sel_{\mathfrak{D}(\chi)}(L_\infty) = H^1\big(\Gal{K_\Sigma}{L_\infty},\mathfrak{D}(\chi)\big)$. If the character $\chi$ is non-trivial, we have $H^0\big(\Gal{K_\Sigma}{L_\infty},\mathfrak{D}(\chi)\big)=0$. In this case when $\chi$ is even and non-trivial, by Lemma \ref{diff-maps-iso}, for each $i \geq 1$ and every subgroup $P$ of $P_G$, we  have
\begin{align} \label{even-global}
H^i\bigg(P, \ H^1\big(\Gal{K_\Sigma}{L_\infty},\mathfrak{D}(\chi)\big)\bigg) = H^i\bigg(\Gal{L_\infty}{F_\infty}, \ H^1\big(\Gal{K_\Sigma}{L_\infty},\mathfrak{D}(\chi)\big)\bigg) = 0.
\end{align}

So by Proposition \ref{greenberg-freeness}, when $\chi$ is even and non-trivial, the $\Lambda[G]$-module $\X$ has a free resolution of length one. \\

Finally, let us suppose that the character $\chi$ is trivial, so that $\mathfrak{D}(\chi) = \frac{\Q_p}{\Z_p}$. In this case, we will choose $P$ (which is isomorphic to $\Gal{L_\infty}{F_\infty}$) so that $P$ is a cyclic group of order $p$. This is possible due to our assumption that $p$ divides the Galois group $G$ (and due to Cauchy's theorem). Now, we have the following sequence of isomorphisms for each $i \geq 1$ (the second isomorphism uses the fact that we have  chosen $\Gal{L_\infty}{F_\infty}$ to be cyclic):
\begin{align*}
H^{2i-1}\bigg(\Gal{L_\infty}{F_\infty}, \ H^1\left(\Gal{K_\Sigma}{L_\infty},\mathfrak{D}(\chi)\right)\bigg) & \cong H^{2i+1}\bigg(\Gal{L_\infty}{F_\infty}, \ H^0\left(\Gal{K_\Sigma}{L_\infty},\frac{\Q_p}{\Z_p} \right)\bigg) \\& \notag \cong H^1\bigg(\Gal{L_\infty}{F_\infty}, \ H^0\left(\Gal{K_\Sigma}{L_\infty},\frac{\Q_p}{\Z_p} \right)\bigg) \\ & \notag \cong \Hom\left(\Gal{L_\infty}{F_\infty}, \frac{\Q_p}{\Z_p}\right) \cong  \frac{\Z}{p\Z}.
\end{align*}

The isomorphisms given in the equation above and Proposition \ref{greenberg-freeness} let us conclude that if the character $\chi$ is even and trivial, the $\Lambda[G]$-module $\X$ does not have a free resolution of length one. This completes the proof of Theorem~\ref{evenprojectivity}.

\end{proof}

\subsection{Local cohomology groups and Proof of Theorem \ref{oddprojectivity}}  \label{localcohosection}

The whole of Section \ref{localcohosection} will be devoted to the proof of Theorem \ref{oddprojectivity}, which we state below.
\odprojectivity*

Let $P_G$ denote a $p$-Sylow subgroup of $G$. Let $P$ denote a subgroup of $P_G$. By Galois theory, we can identify $P$ with $\Gal{L_\infty}{F_\infty}$, for some field $F_\infty$ such that $K_\infty \subset L_\infty^{P_G} \subset F_\infty \subset L_\infty$. We have
\begin{align*}
P \cong \Gal{L_\infty}{F_\infty}.
\end{align*}
Fix a prime $\nu$ in $\Sigma_p(F_\infty)$. Suppose $\omega_1,\ldots,\omega_n$ denote all the primes in $L_\infty$ lying above the prime $\nu$ in $F$. Let $\omega(\nu)$ equal $\omega_1$. Let $P_{\omega(\nu)}$ denote the decomposition group inside $P$, corresponding to the prime $\omega(\nu)$ lying above $\nu$. We have the following isomorphism:
\begin{align*}
\prod_{\omega_i \mid \nu} H^1\left(I_{\omega_i}, \mathfrak{D}(\chi)\right)^{\Gamma_{\omega_i}} \cong \Ind_{{}_{P_{\omega(\nu)}}}^{P}\bigg( H^1\left(I_{\omega(\nu)}, \mathfrak{D}(\chi)\right)^{\Gamma_{\omega(\nu)}}\bigg)
\end{align*}
Shapiro's lemma then lets us deduce the following isomorphism:
\begin{align} \label{shapiro-iso}
 H^i\left(P, \ \prod_{\omega_i \mid \nu} H^1\left(I_{\omega_i}, \mathfrak{D}(\chi)\right)^{\Gamma_{\omega_i}}\right) \cong H^i\left(P_{\omega(\nu)}, H^1\left(I_{\omega(\nu)}, \mathfrak{D}(\chi)\right)^{\Gamma_{\omega(\nu)}} \right).
\end{align}

Corollary 3.2.3 in Greenberg's work \cite{greenberg2010surjectivity}, along with the observation that $H^1\left(I_{\omega(\nu)}, \mathfrak{D}(\chi)\right)^{\Gamma_{\omega(\nu)}}$ is a quotient of the local cohomology group $H^1\left(G_{\omega(\nu)}, \mathfrak{D}(\chi)\right)$ and the fact that $\Sigma$ contains primes above a finite prime number $\nu$ not lying above $p$, let us conclude that the map $\phi^{\Sigma_0}_{\mathfrak{D}(\chi),\text{odd}}$ is surjective. The fact that the map $\phi^{\Sigma_0}_{\mathfrak{D}(\chi),\text{odd}}$ is  surjective crucially relies on the fact that it is a global-to-local map defining the non-primitive Selmer group.  We have the following short exact sequence of $\Lambda[G]$-modules:
\begin{align*}
0 \rightarrow \Sel_{\mathfrak{D}(\chi)}(L_\infty) \rightarrow H^1\left(\Gal{K_\Sigma}{L_\infty}, \mathfrak{D}(\chi)\right) \xrightarrow {\phi^{\Sigma_0}_{\mathfrak{D}(\chi),\text{odd}}} \prod_{\eta \in \Sigma_p({L_\infty}),  } H^1\left(I_\eta,\mathfrak{D}(\chi)\right)^{\Gamma_\eta} \rightarrow 0.
\end{align*}
Let us apply the long exact sequence in Galois cohomology, for the group $P$. Lemma (\ref{odd-global}) let us obtain the following isomorphism, for each $i \geq 1$:
\begin{align} \label{first-obtain-iso}
\prod_{\eta \in \Sigma_p(L_\infty)} H^{i}\left(P,H^1(I_\eta,\mathfrak{D}(\chi))^{\Gamma_\eta}\right) \cong H^{i+1} \left(P, \  \Sel_{\mathfrak{D}(\chi)}\left(L_\infty\right)\right)
\end{align}
The product in equation (\ref{first-obtain-iso}) is indexed by the primes lying above $p$ in $L_\infty$. Equation (\ref{shapiro-iso}) allows us to rewrite the isomorphism in (\ref{first-obtain-iso}) as follows (where the product is now indexed by the primes lying above $p$ in $F_\infty$):
\begin{align} \label{compare-selmer-local}
\prod_{\nu \in \Sigma_p(F_\infty)}  H^i\left(P_{\omega(\nu)}, H^1\left(I_{\omega(\nu)}, \mathfrak{D}(\chi)\right)^{\Gamma_{\omega(\nu)}} \right) \cong H^{i+1} \left(P, \  \Sel_{\mathfrak{D}(\chi)}\left(L_\infty\right)\right), \qquad \forall \ i \geq 1.
\end{align}

Let $G_{\omega(\nu)}$ (and $G_\nu$ respectively) denote the decomposition group lying inside $\Gal{K_\Sigma}{L_\infty}$ (and  $\Gal{K_\Sigma}{F_\infty}$ respectively) corresponding to the prime $\omega(\nu)$ (and $\nu$ respectively). Let $I_{\omega(\nu)}$ (and $I_\nu$ respectively) denote the inertia subgroup inside $G_{\omega(\nu)}$ (and $G_\nu$ respectively). Let $\Gamma_{\omega(\nu)}$ (and $\Gamma_\nu$ respectively) denote the quotient $\frac{G_{\omega(\nu)}}{I_{\omega(\nu)}}$ (and $\frac{G_\nu}{I_\nu}$ respectively).  We have the following natural maps:
\begin{align} \label{natural-decompositions}
G_{\omega(\nu)} \hookrightarrow G_{\nu}, \qquad P_{\omega(\nu)} \cong \frac{G_{\nu}}{G_{\omega(\nu)}}.
\end{align}
The $p$-cohomological dimensions of both the groups $G_\nu$ and $G_{\omega(\nu)}$ equal one (see the discussion on Page 25 of \cite{greenberg2008introduction}). As a result,
\begin{align*}
H^j\left(G_{\omega(\nu)}, \mathfrak{D}(\chi)\right) = 0, \quad H^j\left(G_\nu, \mathfrak{D}(\chi) \right) =0, \quad \forall j \geq 2.
\end{align*}
An argument (similar to the one used to establish Lemma \ref{diff-maps-iso}) involving the spectral sequence
\begin{align*}
H^{i}\left(P_{\omega(\nu)}, H^j\left(G_{\omega(\nu)}, \mathfrak{D}(\chi) \right) \right) \implies H^{i+j}\left(G_\nu, \mathfrak{D}(\chi)\right)
\end{align*}
then allows us obtain the following natural isomorphism for all $i \geq 1$ (via the differential maps):
 \begin{align} \label{skip-two-iso-local}
H^{i}\left(P_{\omega(\nu)}, H^1\left(G_{\omega(\nu)}, \mathfrak{D}(\chi) \right) \right) \cong H^{i+2}\left(P_{\omega(\nu)}, H^0\left(G_{\omega(\nu)}, \mathfrak{D}(\chi) \right) \right).
\end{align}
We will now complete the proof of Theorem \ref{oddprojectivity} by considering the following cases:
\begin{itemize}
\item (Condition I) When $H^0\left(G_{\omega(\nu)}, \mathfrak{D}(\chi)\right)=0$.
\item (Condition II) When $\omega(\nu)$ is tamely ramified in the extension~$L_\infty/K_\infty$.
\item There exists a prime $\omega \in \Sigma_p(L_\infty)$ that doesn't satisfy both conditions \ref{condition1}~and~\ref{condition2}.
\end{itemize}
These cases are considered in Sections \ref{first-odd-case}, \ref{second-odd-case} and \ref{third-odd-case}. The headings in each of these sections highlight the assumption which we will be working with.

\subsubsection{When $H^0\left(G_{\omega(\nu)}, \mathfrak{D}(\chi)\right)=0$}\label{first-odd-case} \mbox{}

Consider the following inflation-restriction short exact sequence that is $P_{\omega(\nu)}$-equivariant:
 \begin{align*}
0 \rightarrow H^1\bigg(\Gamma_{\omega(\nu)},H^0\left(I_{\omega(\nu)}, \mathfrak{D}(\chi)\right)\bigg) \rightarrow H^1\left(G_{\omega(\nu)}, \mathfrak{D}(\chi)\right) \rightarrow H^1\left(I_{\omega(\nu)}, \mathfrak{D}(\chi)\right)^{\Gamma_{\omega(\nu)}}\rightarrow 0.
\end{align*}
We first claim that $H^1\bigg(\Gamma_{\omega(\nu)},H^0\left(I_{\omega(\nu)}, \mathfrak{D}(\chi)\right)\bigg)$ equals zero.  It suffices to prove the claim when $H^0\left(I_{\omega(\nu)}, \mathfrak{D}(\chi)\right) \neq 0$. The reasoning for this is similar to the one given in Observation \ref{non-trivial-residual-obs}. If $H^0\left(I_{\omega(\nu)}, \mathfrak{D}(\chi)\right) \neq 0$, then $$H^0\left(I_{\omega(\nu)}, \mathfrak{D}(\chi)\right) \cong \mathfrak{D}(\chi)= \frac{\Q_p(\chi)}{\Z_p(\chi)}.$$
The group $\Gamma_{\omega(\nu)}$ is topologically generated by Frobenius (denoted by $\Frob_{\omega(\nu)}$). We have the following exact sequence:
{\footnotesize
\begin{align*}
0 \rightarrow  H^1\bigg(\Gamma_{\omega(\nu)}, H^0\left(I_{\omega(\nu)},\mathfrak{D}(\chi)\right)\bigg)^\vee \rightarrow \underbrace{H^0\left(I_{\omega(\nu)}, \mathfrak{D}(\chi)\right)^\vee}_{\Z_p\left(\chi^{-1}\right)} \xrightarrow {\Frob_{\omega(\nu)}-1}\underbrace{H^0\left(I_{\omega(\nu)}, \mathfrak{D}(\chi)\right)^\vee}_{\Z_p\left(\chi^{-1}\right)} \rightarrow   \underbrace{H^0\left(G_{\omega(\nu)}, \mathfrak{D}(\chi) \right)^\vee}_{=0} \rightarrow 0.
\end{align*}}
A surjective endomorphism of a finitely generated free $\Z_p$-module must be an isomorphism. We can now conclude that {\small $H^1\bigg(\Gamma_{\omega(\nu)}, H^0\left(I_{\omega(\nu)},\mathfrak{D}(\chi)\right)\bigg)=0$}. So, we have the following isomorphism that is $P_{\omega(\nu)}$-equivariant: \begin{align}
H^1\left(G_{\omega(\nu)}, \mathfrak{D}(\chi)\right) \cong H^1\left(I_{\omega(\nu)}, \mathfrak{D}(\chi)\right)^{\Gamma_{\omega(\nu)}}.
\end{align}
This lets us obtain the following isomorphism for all $i \geq 0$:

\begin{align} \label{skip-two-first-case}
H^i\bigg(P_{\omega(\nu)}, \ H^1\left(G_{\omega(\nu)}, \mathfrak{D}(\chi)\right)\bigg) \cong H^i\left(P_{\omega(\nu)}, \ H^1\left(I_{\omega(\nu)}, \mathfrak{D}(\chi)\right)^{\Gamma_{\omega(\nu)}}\right).
\end{align}

Equation (\ref{skip-two-iso-local}) lets us conclude that
 \begin{align}\label{zero-first-case-odd} H^{i}\left(P_{\omega(\nu)}, H^1\left(G_{\omega(\nu)}, \mathfrak{D}(\chi) \right) \right) \cong H^{i+2}\left(P_{\omega(\nu)}, H^0\left(G_{\omega(\nu)}, \mathfrak{D}(\chi) \right) \right) = 0, \qquad \forall \ i \geq 1.
\end{align}

Combining (\ref{skip-two-first-case}), (\ref{zero-first-case-odd}), we obtain the following equality for all $i \geq 1$:
 \begin{align} \label{main-comp-first-case}
H^i\left(P_{\omega(\nu)}, \ H^1\left(I_{\omega(\nu)}, \mathfrak{D}(\chi)\right)^{\Gamma_{\omega(\nu)}}\right) & \cong H^{i}\left(P_{\omega(\nu)}, H^1\left(G_{\omega(\nu)}, \mathfrak{D}(\chi) \right) \right) \\ & \notag \cong H^{i+2}\left(P_{\omega(\nu)}, H^0\left(G_{\omega(\nu)}, \mathfrak{D}(\chi) \right) \right) = 0.
\end{align}

\subsubsection{When $\omega(\nu)$ is tamely ramified in the extension~$L_\infty/K_\infty$}\label{second-odd-case} \mbox{}

We would like to start with the following observation. The Galois group $\Gal{L_\infty}{F_\infty}$ is a $p$-group. If $\omega(\nu)$ is tamely ramified in the extension $L_\infty/K_\infty$, then prime  $\omega(\nu)$ in $L_\infty$ (lying over the prime $\nu$ in $F_\infty$) must remain unramified in the extension $L_\infty/F_\infty$. Therefore,  in addition to the maps in (\ref{natural-decompositions}), we have the natural isomorphisms:
\begin{align*}
I_{\omega(\nu)} \stackrel{\cong}{\hookrightarrow} I_{\nu}, \qquad \Gamma_\nu \cong \frac{\Gamma_\omega}{P_\omega}.
\end{align*}
Note that the $p$-cohomological dimensions of the groups $\Gamma_{\omega(\nu)}$ and $\Gamma_\nu$ equal one (these groups are isomorphic to $\hat{\Z}$).  For any $\Gamma_\nu$-module $\mathcal{M}$, an argument involving the spectral sequence (which is similar to the one used to establish  Lemma \ref{diff-maps-iso}),
\begin{align*}
H^{i}\bigg(P_{\omega(\nu)}, H^j\left(\Gamma_{\omega(\nu)}, \mathcal{M} \right) \bigg) \implies H^{i+j} \left( \Gamma_\nu,  \mathcal{M} \right)
\end{align*}
allows us obtain the following natural isomorphism (via the differential maps):
\begin{align} \label{skip-two-iso-local-second}
H^{i}\bigg(P_{\omega(\nu)}, H^1\left(\Gamma_{\omega(\nu)},  \mathcal{M} \right) \bigg) \cong H^{i+2}\bigg(P_{\omega(\nu)}, H^0\left(\Gamma_{\omega(\nu)}, \mathcal{M} \right) \bigg), \qquad \forall \ i \geq 1.
\end{align}
Furthermore, the $p$-cohomological dimension of the inertia group $I_{\omega(\nu)}$ (in addition to the decomposition group $G_{\omega(\nu)}$ and the quotient group $\Gamma_{\omega(\nu)}$) also equals one. See the discussion on Page 25 of \cite{greenberg2008introduction}. Analysing the spectral sequence
\begin{align*}
H^i\bigg(\Gamma_{\omega(\nu)},H^j\left(I_{\omega(\nu)}, \mathfrak{D}(\chi)\right)\bigg) \implies H^{i+j} \left( G_{\omega(\nu)},  \mathfrak{D}(\chi)\right)\end{align*}
now lets us deduce that
\begin{align} \label{brauer-group-zero}
H^1\bigg(\Gamma_{\omega(\nu)}, H^1\left(I_{\omega(\nu)},\mathfrak{D}(\chi)\right) \bigg) =0. \end{align}
To deduce equation (\ref{brauer-group-zero}), we do not use the condition that $\omega(\nu)$ is tamely ramified in the extension~$L_\infty/K_\infty$.

Let $\mathcal{M}$ equal $H^1\left(I_{\omega(\nu)}, \mathfrak{D}(\chi)\right)$. Combining (\ref{skip-two-iso-local-second}) and (\ref{brauer-group-zero}), we obtain the following equality for all $i \geq 3$:
\begin{align} \label{main-comp-second-case}
H^i\left(P_{\omega(\nu)}, \ H^1\left(I_{\omega(\nu)}, \mathfrak{D}(\chi)\right)^{\Gamma_{\omega(\nu)}}\right) \cong H^{i-2}\left(P_{\omega(\nu)}, H^1\left(\Gamma_{\omega(\nu)}, \  H^1\left(I_{\omega(\nu)}, \mathfrak{D}(\chi)\right)\right)\right) =0.
\end{align}
The arguments in Sections \ref{first-odd-case} and \ref{second-odd-case} have the following implications towards Theorem~\ref{oddprojectivity}. Suppose every prime $\omega$ in $\Sigma_p(L_\infty)$ satisfies one of the following conditions:
\begin{enumerate}[label=(\Roman*), ref=\Roman*]
\item\label{condition1} $H^0\left(G_\omega, \mathfrak{D}(\chi)\right)=0$
\item\label{condition2} $\omega$ is tamely ramified in the extension $L_\infty/K_\infty$.
\end{enumerate}
Combining (\ref{compare-selmer-local}), (\ref{main-comp-first-case}) and (\ref{main-comp-second-case}), we obtain the following equality for all $i \geq 4$:
\begin{align}
H^{i} \left(P, \  \Sel_{\mathfrak{D}(\chi)}\left(L_\infty\right)\right) \cong \prod_{\nu \in \Sigma_p(F_\infty)}  H^{i-1}\left(P_{\omega(\nu)}, H^1\left(I_{\omega(\nu)}, \mathfrak{D}(\chi)\right)^{\Gamma_{\omega(\nu)}} \right) =0.
\end{align}
As a result, Proposition \ref{greenberg-freeness} now lets us conclude that the $\Lambda[G]$-module $\X$ has a free resolution of length one if every prime $\omega \in \Sigma_p(L_\infty)$ satisfies Condition \ref{condition1} or \ref{condition2}.

\subsubsection{There exists a prime $\omega \in \Sigma_p(L_\infty)$ that doesn't satisfy both conditions \ref{condition1}~and~\ref{condition2}}\label{third-odd-case} \mbox{}

For the rest of this section, we will work under the assumption given in the heading of this subsection. Under this assumption, we will establish that the $\Lambda[G]$-module $\X$ does not have a free resolution of length one.

Since the prime $\omega$ doesn't satisfy Condition \ref{condition1}, we have $H^0\left(G_\omega, \mathfrak{D}(\chi)\right)\neq 0$. As we argued earlier (see Observation \ref{non-trivial-residual-obs}), this leads us to assert that the restriction of the character $\chi$ to $G_{\omega}$ is trivial and that
 \begin{align} \label{third-case-trivial-action}
H^0\left(I_\omega, \mathfrak{D}(\chi)\right) \cong H^0\left(G_\omega, \mathfrak{D}(\chi)\right) \cong \frac{\Q_p}{\Z_p}.\end{align}
Furthermore, since the prime $\omega$ doesn't satisfy Condition \ref{condition2}, the prime $\omega$ is wildly ramified in the extension $L_\infty/K_\infty$. As a result, $p$ divides the order of the inertia subgroup inside $\Gal{L_\infty}{K_\infty}$, corresponding to the prime $\omega$. We choose $P$ to be a cyclic subgroup of order $p$ inside this inertia subgroup (this is possible due to Cauchy's theorem). As we did in the earlier sections, we will identify $P$ with $\Gal{L_\infty}{F_\infty}$, for some field $F_\infty$ satisfying $K_\infty \subset F_\infty \subset L_\infty$. We let $\nu$ denote the prime in $F_\infty$ lying below $\omega$. To be consistent with our notations in the earlier sections, we let $\omega(\nu)$ denote the prime $\omega$. We have the following natural maps:
\begin{align*}
P_{\omega(\nu)} \stackrel{=}{\hookrightarrow} P \cong \frac{\Z}{p\Z}
\end{align*}
Since the prime $\omega(\nu)$ is totally ramified in the extension $L_\infty/F_\infty$, we have a natural  inclusion of inertia groups $I_{\omega(\nu)} \hookrightarrow I_\nu$. We also have the following natural isomorphisms:
 \begin{align}
 P_{\omega(\nu)} \cong \frac{I_\nu}{I_\omega(\nu)}, \qquad \Gamma_{\omega(\nu)} \cong \Gamma_\nu.
\end{align}
The short exact sequences
 \begin{align*}
0 \rightarrow \underbrace{\frac{I_{\nu}}{I_{\omega(\nu)}}}_{\cong P_{\omega(\nu)}} \rightarrow \frac{G_{\nu}}{I_{\omega(\nu)}} \rightarrow \underbrace{\frac{G_{\nu}}{I_{\nu}}}_{\cong \Gamma_{\nu}} \rightarrow 0, \qquad 0 \rightarrow \underbrace{\frac{G_{\omega(\nu)}}{I_{\omega(\nu)}}}_{\cong \Gamma_{\omega(\nu)}} \rightarrow \frac{G_{\nu}}{I_{\omega(\nu)}} \rightarrow \underbrace{\frac{G_{\nu}}{G_{\omega(\nu)}}}_{\cong P_{\omega(\nu)}} \rightarrow 0, \end{align*}
allow us to view both $P_{\omega(\nu)}$ and $\Gamma_{\omega(\nu)}$ as normal subgroups of $\frac{G_{\nu}}{I_{\omega(\nu)}}$. As a result,
 \begin{align*}
\frac{G_{\nu}}{I_{\omega(\nu)}} \cong P_{\omega(\nu)} \times \Gamma_{\omega(\nu)}.
\end{align*}
Equation (\ref{brauer-group-zero}) tells us that {\small $
H^1\bigg(\Gamma_{\omega(\nu)}, H^1\left(I_{\omega(\nu)},\mathfrak{D}(\chi)\right) \bigg) =0$}. So, we have the following exact sequence that is $P_{\omega(\nu)} \times \Gamma_{\omega(\nu)}$-equivariant:
\begin{align*}
0 \rightarrow H^0\bigg(\Gamma_{\omega(\nu)}, H^1\left(I_{\omega(\nu)},\mathfrak{D}(\chi)\right) \bigg) \rightarrow H^1\left(I_{\omega(\nu)},\mathfrak{D}(\chi)\right) \xrightarrow {\Frob_{\omega(\nu)}-1} H^1\left(I_{\omega(\nu)},\mathfrak{D}(\chi)\right) \rightarrow 0.
\end{align*}
Consider the following long exact sequence in group cohomology  (for the group $P_{\omega(\nu)}$):
{\footnotesize \begin{align} \label{long-exact-seq}
\ldots \rightarrow  H^i\bigg(P_{\omega(\nu)}, H^1\left(I_{\omega(\nu)},\mathfrak{D}(\chi)\right)^{\Gamma_{\omega(\nu)}} \bigg) \rightarrow H^i\bigg(P_{\omega(\nu)}, H^1\left(I_{\omega(\nu)},\mathfrak{D}(\chi)\right)\bigg) \xrightarrow {\Frob_{\omega(\nu)}-1} H^i\bigg(P_{\omega(\nu)},H^1\left(I_{\omega(\nu)},\mathfrak{D}(\chi)\right)\bigg) \rightarrow \ldots
\end{align}}
For each $i \geq 1$, we have the following commutative diagram:
 \begin{align} \label{comm-diagram}
\xymatrix{
 H^i\bigg(P_{\omega(\nu)}, H^1\left(I_{\omega(\nu)},\mathfrak{D}(\chi)\right)\bigg) \ar[d]^{\cong}\ar[rr]^{ \Frob_{\omega(\nu)}-1}&&  H^i\bigg(P_{\omega(\nu)}, H^1\left(I_{\omega(\nu)},\mathfrak{D}(\chi)\right) \bigg)  \ar[d]^{\cong }\\
H^{i+2}\bigg(P_{\omega(\nu)}, H^0\left(I_{\omega(\nu)},\mathfrak{D}(\chi)\right)\bigg) \ar[rr]^{\Frob_{\omega(\nu)}-1} && H^{i+2}\bigg(P_{\omega(\nu)}, H^0\left(I_{\omega(\nu)},\mathfrak{D}(\chi)\right)\bigg)
}
\end{align}
Equation (\ref{third-case-trivial-action}) tells us that $H^0\left(I_{\omega(\nu)},\mathfrak{D}(\chi)\right) \cong \frac{\Q_p}{\Z_p}$. Note that the group $\Gamma_{\omega(\nu)}$, which is topologically generated by $\Frob_{\omega(\nu)}$, acts trivially on both $P_{\omega(\nu)}$ and $\frac{\Q_p}{\Z_p}$. As a result, for each $i \geq 1$, the horizontal maps in the commutative diagram (\ref{comm-diagram}) (given by $\Frob_{\omega(\nu)}-1$) are the zero maps. Since $P_{\omega(\nu)}$ is a cyclic group of order $p$, for all $i \geq 1$,
 \begin{align*}
H^{2i-1}\left(P_{\omega(\nu)}, \frac{\Q_p}{\Z_p}\right) \cong \Hom\left(P_{\omega(\nu)}, \frac{\Q_p}{\Z_p}\right) \cong \frac{\Z}{p\Z}, \qquad H^{2i}\left(P_{\omega(\nu)}, \frac{\Q_p}{\Z_p}\right) = 0.
\end{align*}
Combining these observations pertaining to the long exact sequence (\ref{long-exact-seq}) and the commutative diagram (\ref{comm-diagram}) with the fact that $P_{\omega(\nu)}$ is a cyclic group leads us to conclude that
\begin{align*}
H^i\left(P_{\omega(\nu)}, H^1\left(I_{\omega(\nu)}, \mathfrak{D}(\chi)\right)^{\Gamma_{\omega(\nu)}} \right) \cong \frac{\Z}{p\Z}, \qquad \forall \ i \geq 1.
\end{align*}
Equation (\ref{compare-selmer-local}) now let us deduce that for all $i\geq 2$, we have the following surjection:
\begin{align*}
H^{i} \left(P, \  \Sel_{\mathfrak{D}(\chi)}\left(L_\infty\right)\right) \twoheadrightarrow \frac{\Z}{p\Z}.
\end{align*}
As a result, Proposition \ref{greenberg-freeness} lets us make the following deduction: the $\Lambda[G]$-module $\X$ does not have a free resolution of length one if there exists a prime $\omega \in \Sigma_p(L_\infty)$ that does not satisfiy both Condition \ref{condition1} and \ref{condition2}.

Our observations in Sections \ref{first-odd-case}, \ref{second-odd-case} and \ref{third-odd-case} now complete the proof of Theorem \ref{oddprojectivity}.

\section{Elliptic curve with a cyclic $p^2$ isogeny and proof of Theorem \ref{elliptic-curves}} \label{ellipticcurves-section}

\subsection{Setup of Theorem \ref{elliptic-curves}}
\mbox{}

Let $E$ be an elliptic curve defined over $\Q$ with good ordinary or split multiplicative reduction at $p$. Let $\Phi : E \rightarrow E'$ be a cyclic isogeny, defined over $\Q$, of degree $p^2$. Let $\tilde{\Phi} : E' \rightarrow E$ denote the dual isogeny (which is also a cyclic isogeny over $\Q$ of degree $p^2$). Let us enlarge the set $\Sigma$ to contain all the primes of bad reduction for $E$. The non-primitive Selmer group $\Sel_{E[p^\infty]}(\Q_\infty)$ associated to $E$ (introduced in  the work of Greenberg and Vatsal \cite{greenberg2000iwasawa}) is given below:
\begin{align*}
\Sel_{E[p^\infty]}(\Q_\infty) &= \ker\bigg( H^1\left(\Gal{\Q_\Sigma}{\Q_\infty},E[p^\infty]\right) \rightarrow H^1(I_\eta,\mathfrak{A})^{G_\eta/I_\eta} \bigg).
\end{align*}
Here, $G_\eta$ and $I_\eta$ denote the decomposition and inertia subgroup for the unique prime $\eta$ in $\Q_\infty$ lying above $p$. If $E$ has good reduction at $p$  then we let $\mathfrak{A}$ equal the $\Gal{\overline{\Q}_p}{\Q_p}$-module $\overline{E}[p^\infty]$, the $p$-power torsion points on the reduced elliptic curve $\overline{E}$. If $E$ has split multiplicative reduction at $p$, then $\mathfrak{A}$ is defined to be $\Q_p/\Z_p$ with the trivial action of $\Gal{\overline{\Q}_p}{\Q_p}$. \\

Both $\ker(\Phi)$ and $\ker(\tilde{\Phi})$ are cyclic groups of order $p^2$. Without loss of generality, we shall henceforth assume that the action of $\Gal{\Q_\Sigma}{\Q}$ on $\ker(\Phi)$ is even (otherwise we could simply consider the curve $E'$ and the dual isogeny). We have the following natural characters:
\begin{align*}
& \phi : \Gal{\Q_\Sigma}{\Q} \rightarrow \mathrm{Aut}\left(\ker(\Phi)\right) \cong \left(\Z/p^2\Z\right)^\times, \qquad && \tilde{\phi}: \Gal{\Q_\Sigma}{\Q} \rightarrow \mathrm{Aut}\left(\ker\left(\tilde{\Phi}\right)\right) \cong \left(\Z/p^2\Z\right)^\times \\
& \chi_\phi : \Gal{\Q_\Sigma}{\Q} \rightarrow \mathrm{Aut}\left(\ker(\Phi)[p]\right) \cong \left(\Z/p\Z\right)^\times, \qquad && \chi_{\tilde{\phi}} : \Gal{\Q_\Sigma}{\Q} \rightarrow \mathrm{Aut}\left(\ker\left(\tilde{\Phi}\right)[p]\right) \cong \left(\Z/p\Z\right)^\times,
\end{align*}

The Weil pairing gives us the following equality of characters:
\begin{align}\label{weil-pairing}
\phi \tilde{\phi} = \chi_p \ (\text{mod } p^2), \qquad \chi_\phi \chi_{\tilde{\phi}}= \chi_p \ (\text{mod } p).
\end{align}
Here $\chi_p:\Gal{\Q_\Sigma}{\Q} \rightarrow \Z_p^\times$ denotes the $p$-adic cyclotomic~character. Consider the following fields and the associated Galois groups:
\begin{align*}
& \Q_{\phi} :=\overline{\Q}^{\ker(\phi)}, \quad  && \Q_{\tilde{\phi}} :=\overline{\Q}^{\ker(\tilde{\phi})}, \quad && \Q_{\chi_\phi} :=\overline{\Q}^{\ker(\chi_\phi)}, \quad && \Q_{\chi_{\tilde{\phi}}} :=\overline{\Q}^{\ker(\chi_{\tilde{\phi}})}.\\
&   &&  && \Delta_{\chi_\phi} :=\Gal{\Q_{\chi_\phi}}{\Q}, \quad && \Delta_{\chi_{\tilde{\phi}}} :=\Gal{\Q_{\chi_{\tilde{\phi}}}}{\Q}.
\end{align*}

Let $G_\phi$ and $G_{\tilde{\phi}}$ denote the $p$-Sylow subgroups of $\Gal{\Q_{\phi}}{\Q}$ and $\Gal{\Q_{\tilde{\phi}}}{\Q}$ respectively. Note that the groups $G_\phi$ and $G_{\tilde{\phi}}$ are of order dividing $p$. Also note that the groups $\Delta_{\chi_\phi}$ and $\Delta_{\chi_{\tilde{\phi}}}$ are of order dividing $p-1$. One can view $G_\phi$ and $G_{\tilde{\phi}}$  as quotients of $\Gal{\Q_{\phi}}{\Q}$ and $\Gal{\Q_{\tilde{\phi}}}{\Q}$, respectively, as well. We can consider the following field diagrams:
\begin{center}
\begin{tikzpicture}[node distance = 1cm, auto]
      \node (Q) {$\Q$};
      \node (Qpiep) [above of=Q, left of=K] {$L_\phi$};
      \node (Qdelep) [above of=Q, right of=K] {$\Q_{\chi_\phi}$};
      \node (Qep) [above of=K, node distance = 2cm] {$\Q_{\phi}$};
      \draw[-] (Q) to node {$G_\phi$} (Qpiep);
      \draw[-] (Q) to node [swap] {$\Delta_{\chi_\phi}$} (Qdelep);
      \draw[-] (Qpiep) to node {$\Delta_{\chi_\phi}$} (Qep);
      \draw[-] (Qdelep) to node [swap] {$G_\phi$} (Qep);

\node (Q2)  [right of = Q, node distance = 5cm] {$\Q$};
      \node (Qpiphi) [above of=Q2, left of = Q2] {$L_{\tilde{\phi}}$};
      \node (Qdelphi) [above of=Q2, right of=Q2] {$\Q_{\chi_{\tilde{\phi}}}$};
      \node (Qphi) [above of=Q2, node distance = 2cm] {$\Q_{\tilde{\phi}}$};
      \draw[-] (Q2) to node {$G_{\tilde{\phi}}$} (Qpiphi);
      \draw[-] (Q2) to node [swap] {$\Delta_{\chi_{\tilde{\phi}}}$} (Qdelphi);
      \draw[-] (Qpiphi) to node {$\Delta_{\chi_{\tilde{\phi}}}$} (Qphi);
      \draw[-] (Qdelphi) to node [swap] {$G_{\tilde{\phi}}$} (Qphi);

      \node (Iso) [right of=Qphi, node distance = 5cm] {   $G_\phi \cong \Gal{L_\phi}{\Q}$.};

      \node (Iso2) [below of=Iso, node distance = 1cm] {   $G_{\tilde{\phi}} \cong \Gal{L_{\tilde{\phi}}}{\Q}$.};

     \end{tikzpicture}
  \end{center}
We will need to impose the following condition (similar to (\ref{field-condn})):
\begin{enumerate}[label=(Non-DG),ref=(Non-DG)]
\item\label{Non-DG} $L_{\phi} \cap \Q_\infty = \Q$.
\end{enumerate}
As we shall see in Corollary \ref{intersection-non-degenerate}, the condition given in (\ref{field-condn}) is automatically satisfied for $L_{\tilde{\phi}}$. That is, $L_{\tilde{\phi}} \cap \Q_\infty = \Q$. To verify \ref{Non-DG} in practice, one can use, for example, Velu's formula (\cite{MR0294345}) to explicitly find a generating polynomial for the field extension $L_\phi$ over $\Q$. \\

For our application, we shall consider the Selmer groups $\Sel_{\mathfrak{D}(\chi_\phi)}(L_{\phi,\infty})$ and $\Sel_{\mathfrak{D}(\chi_{\tilde{\phi}})}(L_{\tilde{\phi},\infty})$ whose Pontryagin duals are finitely generated torsion modules over the completed group rings  $\Lambda[G_\phi]$ and $\Lambda[G_{\tilde{\phi}}]$ respectively. The Pontryagin dual of these Selmer groups are non-primitive Iwasawa modules; they can be denoted by $\X_{\mathfrak{D}(\chi_\phi)}(L_{\phi,\infty})$ and $\X_{\mathfrak{D}(\chi_{\tilde{\phi}})}(L_{\tilde{\phi},\infty})$, following the notations of the introduction. For this section, we will  use the notation involving the Selmer group (instead of using the notations involving $\X$). \\

The restriction of the group homomorphisms
\begin{align}\phi \mid_{G_\phi} : G_{\phi} \rightarrow \left(\Z/p^2\Z\right)^\times, \qquad  \tilde{\phi} \mid_{G_{\tilde{\phi}}} : G_{\tilde{\phi}} \rightarrow \left(\Z/p^2\Z\right)^\times
\end{align}
allow us to consider the following ring homomorphisms:
\begin{align}\label{groupring-homo}
\sigma_{\phi} : \Lambda[G_\phi] \rightarrow \left(\frac{\Z}{p^2\Z}\right)[[\Gamma]] \cong \frac{\Lambda}{(p^2)}, &  \qquad \sigma_{\tilde{\phi}} : \Lambda[G_{\tilde{\phi}}] \rightarrow \left(\frac{\Z}{p^2\Z}\right)[[\Gamma]] \cong \frac{\Lambda}{(p^2)}.
\end{align}

As stated above, we shall work with the following assumption:
\begin{center}
The even character $\chi_\phi$ is ramified at $p$.
\end{center}
Keep in mind that the elliptic curve $E$ has either good ordinary reduction or split multiplicative reduction at $p$. As a result, the semi-simplification of the $\mathbb{F}_p[I_p]$-representation $E[p]$ is a sum of two distinct characters; one of which coincides with the Teichm\"{u}ller character (which is ramified at $p$) and one of which is trivial. Here, $E[p]$ denotes the $p$-power torsion points on the elliptic curve $E$ and $I_p$ denotes the inertia subgroup of $\Gal{\overline{\Q}_p}{\Q_p}$. So, our assumption on the even character is valid if and only if we place the following assumption on the odd character:
\begin{center}
The odd character $\chi_{\tilde{\phi}}$ is unramified at $p$.
\end{center}

\begin{lemma}\label{kernel-isogenies-ramification}
The action of $\Gal{\overline{\Q}_p}{\Q_p}$ on $\ker\left(\Phi\right)$ is totally ramified. The action of $\Gal{\overline{\Q}_p}{\Q_p}$ on $\ker\left(\tilde{\Phi}\right)$ is unramified. In fact, we have the following isomorphism of $\frac{\Z}{p^2\Z}[\Gal{\overline{\Q}_p}{\Q_p}]$-modules:
\begin{align*}
\mathfrak{A}[p^2] \cong \ker\left(\tilde{\Phi}\right).
\end{align*}
\end{lemma}

\begin{proof}
Observe that the semi-simplification of the $\frac{\Z}{p^2\Z}[\Gal{\overline{\Q}_p}{\Q_p}]$-module  $E[p^2]$ is the direct sum
\begin{align*}
\ker(\Phi) \oplus \ker\left(\tilde{\Phi}\right).
\end{align*}
Our assumptions tell us that the action of $\Gal{\overline{\Q}_p}{\Q_p}$ on $\ker\left(\tilde{\Phi}\right)[p]$ is unramified and the action of $\Gal{\overline{\Q}_p}{\Q_p}$ on $\ker(\Phi)[p]$ is ramified. The natural surjection $E[p^2] \twoheadrightarrow \mathfrak{A}[p^2]$ of free $\frac{\Z}{p^2\Z}$-modules, that is $\Gal{\overline{\Q}_p}{\Q_p}$-equivariant, tells us that $\mathfrak{A}[p^2]$ is one of the components appearing in the semi-simplification of $E[p^2]$ over $\frac{\Z}{p^2\Z}[\Gal{\overline{\Q}_p}{\Q_p}]$. Since the action of $\Gal{\overline{\Q}_p}{\overline{\Q}_p}$ on $\mathfrak{A}[p]$ is unramified, we must have the following isomorphism of $\frac{\Z}{p^2\Z}[\Gal{\overline{\Q}_p}{\Q_p}]$-modules:
\begin{align*}
\mathfrak{A}[p^2] \cong \ker(\tilde{\Phi}),
\end{align*}
The fact that the action of $\Gal{\overline{\Q}_p}{\Q_p}$ on $\ker\left(\Phi\right)$ is totally ramified now follows from equality of characters obtained in (\ref{weil-pairing}) using the Weil pairing.
\end{proof}

As an immediate corollary to Lemma \ref{kernel-isogenies-ramification} (in particular, since the action of $\Gal{\overline{\Q}_p}{\Q_p}$ on $\ker\left(\tilde{\Phi}\right)$ is unramified), we have the following corollary:
\begin{corollary} \label{intersection-non-degenerate} $L_{\tilde{\phi}} \cap \Q_\infty = \Q$.
\end{corollary}

The hypotheses of Theorems \ref{oddprojectivity} and  \ref{evenprojectivity} are satisfied for the $\Lambda[G_\phi]$-module $\Sel_{\mathfrak{D}(\chi_\phi)}(L_{\phi,\infty})^\vee$ and the $\Lambda[G_{\tilde{\phi}}]$-module $\Sel_{\mathfrak{D}(\chi_{\tilde{\phi}})}(L_{\tilde{\phi},\infty})^\vee$. They let us deduce that we have the following short exact sequences of $\Lambda[G_\phi]$-modules and $\Lambda[G_{\tilde{\phi}}]$-modules respectively:
 \begin{align}\label{2-classical-ses}
&0 \rightarrow \Lambda[G_\phi]^n \xrightarrow {A_\phi} \Lambda[G_\phi]^n \rightarrow \Sel_{\mathfrak{D}(\chi_\phi)}(L_{\phi,\infty})^\vee  \rightarrow 0, \qquad && A_\phi \in M_n\left(\Lambda[G_\phi]\right).
\\  \notag & 0 \rightarrow \Lambda[G_{\tilde{\phi}}]^m \xrightarrow {A_{\tilde{\phi}}} \Lambda[G_{\tilde{\phi}}]^m \rightarrow \Sel_{\mathfrak{D}(\chi_{\tilde{\phi}})}(L_{\tilde{\phi},\infty})^\vee  \rightarrow 0, \qquad && A_{\tilde{\phi}} \in M_m\left(\Lambda[G_{\tilde{\phi}}]\right).
\end{align}

One can naturally view $\Lambda$ as a subring of $\Lambda[G_\phi]$ and $\Lambda[G_{\tilde{\phi}}]$. Since the groups $G_\phi$ and $G_{\tilde{\phi}}$ are finite, the ring extensions $\Lambda \hookrightarrow \Lambda[G_{\tilde{\phi}}]$ and $\Lambda\hookrightarrow \Lambda[G_{\tilde{\phi}}]$ are integral. One has the following~natural~maps:
\begin{align} \label{natural-compositions}
\Lambda \hookrightarrow \Lambda[G_\phi] \xrightarrow {\sigma_\phi} \frac{\Lambda}{(p^2)}, \qquad \Lambda\hookrightarrow \Lambda[G_{\tilde{\phi}}] \xrightarrow {\sigma_{\tilde{\phi}}} \frac{\Lambda}{(p^2)}.
\end{align}

The rings  $\Z_p[G_\phi]$ and $\Z_p[G_{\tilde{\phi}}]$  are local rings, each of which has a unique maximal ideal containing $p$. The ring $\frac{\Lambda}{(p^2)}$ has a unique minimal prime ideal (which corresponds to the prime ideal $(p)$ in the ring $\Lambda$ and the prime ideal lying above $(p)$ in  $\Lambda[G_\phi]$ and $\Lambda[G_{\tilde{\phi}}]$ respectively). Since the fields $L_\phi$ and $L_{\tilde{\phi}}$ are abelian over $\Q$, a theorem of Ferrero and Washington (\cite{MR528968}) asserts that the $\mu$-invariants of the $\Lambda$-modules $\Sel_{\mathfrak{D}(\chi_\phi)}(L_{\phi,\infty})^\vee$ and $\Sel_{\mathfrak{D}(\chi_{\tilde{\phi}})}(L_{\tilde{\phi},\infty})^\vee$ equal zero. As a result, $\det\left(A_\phi\right)$ and $\det\left(A_{\tilde{\phi}}\right)$ do not belong to the prime ideal lying above $(p)$ in the rings $\Lambda[G_\phi]$ and $\Lambda[G_{\tilde{\phi}}]$ respectively. Note that every zero-divisor in the ring $\frac{\Lambda}{(p^2)}$ is also nilpotent. So, the elements $\det\left(\sigma_\phi\left(A_\phi\right)\right)$ and $\det\left(\sigma_{\tilde{\phi}}\left(A_{\tilde{\phi}}\right)\right)$ in the ring $\frac{\Lambda}{(p^2)}$ are non-zero divisors. By applying Proposition 6 in Chapter III, \S 8 of \cite{bourbaki1998algebra},  one can conclude that tensoring the short exact sequences in (\ref{2-classical-ses}) with $\frac{\Lambda}{(p^2)}$, via the maps $\sigma_\phi$ and $\sigma_{\tilde{\phi}}$ respectively, will give us the following short exact sequences:
\begin{align} \label{2-classical-ses-mod-p}
& 0 \rightarrow \left(\frac{\Lambda}{(p^2)}\right)^n \xrightarrow {\sigma_{\phi}(A_\phi)} \left(\frac{\Lambda}{(p^2)}\right)^n \rightarrow \Sel_{\mathfrak{D}(\chi_\phi)}(L_{\phi,\infty})^\vee \otimes_{\Lambda[G_\phi]} \frac{\Lambda}{(p^2)}   \rightarrow 0.
\\ \notag  & 0 \rightarrow \left(\frac{\Lambda}{(p^2)}\right)^m \xrightarrow {\sigma_{\tilde{\phi}}(A_{\tilde{\phi}})} \left(\frac{\Lambda}{(p^2)}\right)^m \rightarrow \Sel_{\mathfrak{D}(\chi_{\tilde{\phi}})}(L_{\tilde{\phi},\infty})^\vee \otimes_{\Lambda[G_{\tilde{\phi}}]} \frac{\Lambda}{(p^2)}   \rightarrow 0.
\end{align}

In a rather ad-hoc manner, we define the following Selmer groups (which, as one observes, turn out to be  $\frac{\Lambda}{(p^2)}$-modules):
\begin{align*}
& \Sel_{\ker(\Phi)}(\Q_\infty) := H^1\left(\Gal{\Q_\Sigma}{\Q_\infty},\ker(\Phi)\right),\\
& \Sel_{\ker(\tilde{\Phi})}(\Q_\infty) := \ker\bigg(H^1\left(\Gal{\Q_\Sigma}{\Q_\infty}, \ker(\tilde{\Phi})\right) \rightarrow H^1(I_\eta, \ker(\tilde{\Phi}))\bigg).
\end{align*}

Here, $\eta$ denotes the unique prime in $\Q_\infty$ lying above $p$ and $I_\eta$ denotes the corresponding inertia subgroup. We now state a ``control theorem''.
\begin{proposition}
We have the following isomorphism of $\frac{\Lambda}{(p^2)}$-modules:
\begin{align} \label{evencontrol} \Sel_{\ker(\Phi)}(\Q_\infty)^\vee \cong \Sel_{\mathfrak{D}(\chi_\phi)}\left({L_{\phi,\infty}}\right)^\vee \otimes_{\Lambda[G_\phi]} \frac{\Lambda}{(p^2)}, \\ \label{oddcontrol}
\Sel_{\ker\left(\tilde{\Phi}\right)}(\Q_\infty)^\vee \cong \Sel_{\mathfrak{D}\left(\chi_{\tilde{\phi}}\right)}\left({L_{\tilde{\phi},\infty}}\right)^\vee \otimes_{\Lambda[G_{\tilde{\phi}}]} \frac{\Lambda}{(p^2)}.
\end{align}
\end{proposition}

\begin{proof}
We will show that equation (\ref{oddcontrol}) holds. The validity of equation (\ref{evencontrol}) follows similarly. Let us choose a generator, say $g_{\tilde{\phi}}$, of the cyclic group $G_{\tilde{\phi}}$. Let $\vartheta$ equal the element $g_{\tilde{\phi}} - \sigma_{\tilde{\phi}}(g_{\tilde{\phi}})$ in the ring $\Lambda[G_{\tilde{\phi}}]$. One can easily show that $\vartheta$ generates the ideal $\ker\left(\sigma_{\tilde{\phi}}\right)$. To prove  (\ref{oddcontrol}), one needs to establish the following isomorphism:
\begin{align*}\Sel_{\ker\left(\tilde{\Phi}\right)}(\Q_\infty) \stackrel{?}{\cong} \Sel_{\mathfrak{D}\left(\chi_{\tilde{\phi}}\right)}(L_{\tilde{\phi},\infty}) [\vartheta]. \end{align*}

Consider the tautological character $\kappa: \Gal{\Q_\Sigma}{\Q_\infty} \twoheadrightarrow \Gal{L_{\tilde{\phi},\infty}}{\Q_\infty} \cong G_{\tilde{\phi}} \hookrightarrow \Gl_1\left(\Z_p[G_{\tilde{\phi}}]\right).$  We obtain the following short exact sequence of $\Z_p[G_{\tilde{\phi}} ]$-modules that is $\Gal{\Q_\Sigma}{\Q_\infty}$-equivariant:
{\small \begin{align} \label{kappaintroduction}
0 \rightarrow \Z_p[G_{\tilde{\phi}} ]\left(\kappa^{-1} \chi^{-1}_{\tilde{\phi}}\right)  \xrightarrow {\vartheta} \Z_p[G_{\tilde{\phi}} ]\left(\kappa^{-1} \chi^{-1}_{\tilde{\phi}}\right)  \rightarrow \frac{\Z}{p^2\Z}\left(\tilde{\phi}^{-1}\right) \rightarrow 0.
\end{align}}
Here, $\Z_p[G_{\tilde{\phi}} ]\left(\kappa^{-1} \chi^{-1}_{\tilde{\phi}}\right) $ denotes a free $\Z_p[G_{\tilde{\phi}} ]$-module of rank $1$ on which $\Gal{\Q_\Sigma}{\Q_\infty}$ acts via the character $\kappa^{-1} \chi^{-1}_{\tilde{\phi}}$. Note that the group $G_{\tilde{\phi}}$ is isomorphic to the quotient $\frac{\Gal{\Q_\Sigma}{\Q_\infty}}{\Gal{\Q_\Sigma}{L_{\tilde{\phi},\infty}}}$. Let us keep in mind the following isomorphisms that are $\Gal{\Q_\Sigma}{\Q_\infty}$-invariant:
\begin{align*}
\Ind^{\Gal{\Q_\Sigma}{\Q_\infty}}_{\Gal{\Q_\Sigma}{L_{\tilde{\phi},\infty}}} \bigg(\frac{\Q_p(\chi_{\tilde{\phi}})}{\Z_p(\chi_{\tilde{\phi}})}\bigg) & \cong  \frac{\Q_p[G_{\tilde{\phi}}]\left(\kappa \chi_{\tilde{\phi}}\right)}{\Z_p[G_{\tilde{\phi}}]\left(\kappa \chi_{\tilde{\phi}}\right)}  \\ & \cong \frac{\Q_p}{\Z_p} \otimes_{\Z_p}{\Z_p[G_{\tilde{\phi}}]}\left(\kappa \chi_{\tilde{\phi}}\right) \\ &\cong \Hom_{\Z_p}\left(\Z_p, \frac{\Q_p}{\Z_p}\right) \otimes_{\Z_p} \Z_p[G_{\tilde{\phi}}]\left(\kappa \chi_{\tilde{\phi}}\right)    \\ & \cong \Hom_{\Z_p[G_{\tilde{\phi}}]}\left(\Z_p[G_{\tilde{\phi}}], \frac{\Q_p[G_{\tilde{\phi}}]}{\Z_p[G_{\tilde{\phi}}]}\right)\left(\kappa \chi_{\tilde{\phi}}\right) \quad \text{(Theorem 7.11 in Matsumura's book \cite{matsumura1989commutative})} \\ & \cong \Hom_{\Z_p} \left(\Z_p[G_{\tilde{\phi}}]\left(\kappa^{-1} \chi^{-1}_{\tilde{\phi}}\right), \frac{\Q_p}{\Z_p} \right) \quad \text{(since $\Ind^{G_{\tilde{\phi}}}_{\{1\}}$ is right adjoint to $\Res^{G_{\tilde{\phi}}}_{\{1\}}$)}.
\end{align*}

Considering the Pontryagin duals of all the modules appearing in the exact sequence (\ref{kappaintroduction}), we obtain the following exact sequence of discrete modules that is  $\Gal{\Q_\Sigma}{\Q_\infty}$-equivariant:
 \begin{align} \label{short-exact-main-control}
0 \rightarrow \ker\left(\tilde{\Phi}\right) \rightarrow \Ind^{\Gal{\Q_\Sigma}{\Q_\infty}}_{\Gal{\Q_\Sigma}{L_{\tilde{\phi},\infty}}} \bigg(\frac{\Q_p(\chi_{\tilde{\phi}})}{\Z_p(\chi_{\tilde{\phi}})}\bigg) \xrightarrow {\vartheta} \Ind^{\Gal{\Q_\Sigma}{\Q_\infty}}_{\Gal{\Q_\Sigma}{L_{\tilde{\phi},\infty}}} \bigg(\frac{\Q_p(\chi_{\tilde{\phi}})}{\Z_p(\chi_{\tilde{\phi}})}\bigg) \rightarrow 0.
\end{align}

Shapiro's Lemma along with the observations that $\chi_{\tilde{\phi}}$ is odd and $L_{\tilde{\phi},\infty}$ is a totally real field,  tells us that
\begin{align}
&H^0\left(\Gal{\Q_\Sigma}{\Q_\infty},\Ind^{\Gal{\Q_\Sigma}{\Q_\infty}}_{\Gal{\Q_\Sigma}{L_{\tilde{\phi},\infty}}} \bigg(\frac{\Q_p(\chi_{\tilde{\phi}})}{\Z_p(\chi_{\tilde{\phi}})}\bigg) \right) \cong H^0\left(\Gal{\Q_\Sigma}{L_{\tilde{\phi},\infty}}, \frac{\Q_p(\chi_{\tilde{\phi}})}{\Z_p(\chi_{\tilde{\phi}})}  \right) = 0, \\
\notag & H^1\left(\Gal{\Q_\Sigma}{\Q_\infty},\Ind^{\Gal{\Q_\Sigma}{\Q_\infty}}_{\Gal{\Q_\Sigma}{L_{\tilde{\phi},\infty}}} \bigg(\frac{\Q_p(\chi_{\tilde{\phi}})}{\Z_p(\chi_{\tilde{\phi}})}\bigg) \right) \cong H^1\left(\Gal{\Q_\Sigma}{L_{\tilde{\phi},\infty}}, \frac{\Q_p(\chi_{\tilde{\phi}})}{\Z_p(\chi_{\tilde{\phi}})} \right).
\end{align}

Let $\eta$ be the unique prime in $\Q_\infty$ lying above $p$. Observe that since the extension $L_{\tilde{\phi},\infty}/\Q_\infty$ is abelian of order $p$, there either exists exactly one prime or $p$ distinct primes in the field $L_{\tilde{\phi},\infty}$ lying above $\eta$. The extension must be unramified by Lemma \ref{kernel-isogenies-ramification}. Using Shapiro's lemma, we obtain the following isomorphism:

\begin{align} \label{desc-H1-shapiro}
H^1\left(I_\eta, \ \Ind^{\Gal{\Q_\Sigma}{\Q_\infty}}_{\Gal{\Q_\Sigma}{L_{\tilde{\phi},\infty}}} \bigg(\frac{\Q_p(\chi_{\tilde{\phi}})}{\Z_p(\chi_{\tilde{\phi}})}\bigg) \right)
 \cong \left\{ \begin{array}{lc} \bigoplus \limits_{\omega_i \mid \eta} H^1\left(I_{\omega_i},\frac{\Q_p(\chi_{\tilde{\phi}})}{\Z_p(\chi_{\tilde{\phi}})}\right), & \text{if $\eta$  splits completely into $\omega_1 \cdots \omega_p$} \\  \bigoplus \limits_{\text{$p$  copies}} H^1\left(I_{\omega},\frac{\Q_p(\chi_{\tilde{\phi}})}{\Z_p(\chi_{\tilde{\phi}})}\right), & \text{if $\omega \mid \eta$ remains inert}. \end{array} \right.
\end{align}

The discussions in section \ref{localcohosection} (in particular, see Observation \ref{non-trivial-residual-obs}) also tell us that $H^0\left(I_{\omega},\frac{\Q_p(\chi_{\tilde{\phi}})}{\Z_p(\chi_{\tilde{\phi}})}\right)$ is either $0$ or isomorphic to $\frac{\Q_p}{\Z_p}$ as a group, for each $\omega \mid \eta$. In either case, it is a divisible group. As a result, using Shapiro's lemma, one can conclude that $H^0\left(I_\eta, \ \Ind^{\Gal{\Q_\Sigma}{\Q_\infty}}_{\Gal{\Q_\Sigma}{L_{\tilde{\phi},\infty}}} \bigg(\frac{\Q_p(\chi_{\tilde{\phi}})}{\Z_p(\chi_{\tilde{\phi}})}\bigg) \right)$ is also a divisible group. Since the kernel of the map
\begin{align}
H^0\left(I_\eta, \ \Ind^{\Gal{\Q_\Sigma}{\Q_\infty}}_{\Gal{\Q_\Sigma}{L_{\tilde{\phi},\infty}}} \bigg(\frac{\Q_p(\chi_{\tilde{\phi}})}{\Z_p(\chi_{\tilde{\phi}})}\bigg) \right) \xrightarrow {\vartheta} H^0\left(I_\eta, \ \Ind^{\Gal{\Q_\Sigma}{\Q_\infty}}_{\Gal{\Q_\Sigma}{L_{\tilde{\phi},\infty}}} \bigg(\frac{\Q_p(\chi_{\tilde{\phi}})}{\Z_p(\chi_{\tilde{\phi}})}\bigg) \right),
\end{align}
being a subgroup of $\ker\left(\tilde{\Phi}\right)$, is finite, it must  also be surjective. Consider the long exact sequence in group cohomology obtained from the short exact sequence (\ref{short-exact-main-control}), for both groups $\Gal{\Q_\Sigma}{\Q_\infty}$ and $I_\eta$. Combining our various observations lets us obtain the following commutative diagram:
\begin{align*}
\xymatrix{
H^1\left(\Gal{\Q_\Sigma}{\Q_\infty}, \ker\left(\tilde{\Phi}\right) \right) \ar[r]^{\cong \qquad} \ar[d]& H^1\left(\Gal{\Q_\Sigma}{L_{\tilde{\phi},\infty}}, \frac{\Q_p(\chi_{\tilde{\phi}})}{\Z_p(\chi_{\tilde{\phi}})} \right) [\vartheta] \ar[d] \\
H^1\left(I_\eta, \ker\left(\tilde{\Phi}\right) \right)  \ar[r]^{\cong \qquad} & H^1\left(I_\eta, \ \Ind^{\Gal{\Q_\Sigma}{\Q_\infty}}_{\Gal{\Q_\Sigma}{L_{\tilde{\phi},\infty}}} \bigg(\frac{\Q_p(\chi_{\tilde{\phi}})}{\Z_p(\chi_{\tilde{\phi}})}\bigg) \right) [\vartheta]
}
\end{align*}

Considering the kernels of the vertical maps and the description of $H^1\left(I_\eta, \ \Ind^{\Gal{\Q_\Sigma}{\Q_\infty}}_{\Gal{\Q_\Sigma}{L_{\tilde{\phi},\infty}}} \bigg(\frac{\Q_p(\chi_{\tilde{\phi}})}{\Z_p(\chi_{\tilde{\phi}})}\bigg) \right)$ given in (\ref{desc-H1-shapiro}), we obtain the desired isomorphism
\begin{align*}
\Sel_{\ker\left(\tilde{\Phi}\right)}(\Q_\infty) \cong \Sel_{\mathfrak{D}\left(\chi_{\tilde{\phi}}\right)}(L_{\tilde{\phi},\infty}) [\vartheta].
\end{align*}

The proposition follows.
\end{proof}

We will now recall some results of Greenberg-Vatsal \cite{greenberg2000iwasawa} concerning $\Sel_{E[p^\infty]}(\Q_\infty)$. Note that a finitely generated module over a $2$-dimensional regular local ring  has no non-trivial pseudo-null submodules if and only if its projective dimension is less than or equal to $1$. The following proposition is proved in \cite{greenberg2000iwasawa}:

\begin{proposition}[Proposition 2.5 in \cite{greenberg2000iwasawa}]
The $\Lambda$-module  $\Sel_{E[p^\infty]}(\Q_\infty)^\vee$ has no non-zero pseudo-null submodules. As a result,
\begin{align*}
\text{proj dim}_{\Lambda}\left(\Sel_{E[p^\infty]}(\Q_\infty)^\vee\right) \leq 1.
\end{align*}
\end{proposition}

Note that over a commutative local ring, every finitely generated projective module is free. We can consider a free resolution of $\Sel_{E[p^\infty]}(\Q_\infty)^\vee$ as a $\Lambda$-module:
\begin{align}\label{ell-free-res}
0 \rightarrow \Lambda^r \xrightarrow {A_E} \Lambda^r \rightarrow \Sel_{E[p^\infty]}(\Q_\infty) \rightarrow 0.
\end{align}

Here, $A_E$ is an $r \times r$ matrix with entries in $\Lambda$. Note that the characteristic ideal $\text{Char}_{\Lambda}\left(\Sel_{E[p^\infty]}(\Q_\infty)^\vee\right)$ equals the ideal generated by $\det(A_E)$ in $\Lambda$.  By Proposition 5.10 in Greenberg's work \cite{greenberg1999iwasawa}, the $\mu$-invariant of the $\Lambda$-module $\Sel_{E[p^\infty]}(\Q_\infty)^\vee$  is zero. So, the prime number $p$, viewed as an irreducible in the regular local ring $\Lambda$, does not divide $\det(A_E)$. We have the following natural ring homomorphism:
\begin{align*}
\sigma_{p^2}: \Lambda \rightarrow \frac{\Lambda}{(p^2)}.
\end{align*}
By applying Proposition 6 in Chapter III, \S 8 of \cite{bourbaki1998algebra} as we did earlier, one can conclude that tensoring the short exact sequence (\ref{ell-free-res}) with $\frac{\Lambda}{(p^2)}$ gives us the following short exact sequence:
{\small \begin{align} \label{reduction-char-elliptic}
0 \rightarrow \left(\frac{\Lambda}{(p^2)}\right)^r \xrightarrow {\sigma_{p^2}(A_E)} \left(\frac{\Lambda}{(p^2)}\right)^r \rightarrow \Sel_{E[p^\infty]}(\Q_\infty)^\vee \otimes_{\Lambda} \frac{\Lambda}{(p^2)} \rightarrow 0.
\end{align}}
Once again, in a rather ad-hoc manner, we define a Selmer group $\Sel_{E[p^2]}(\Q_\infty)$ for the Galois module $E[p^2]$ over $\Q_\infty$. It turns out to be a module over the ring $\frac{\Lambda}{(p^2)}$.
\begin{align*}
\Sel_{E[p^2]}(\Q_\infty) &:= \ker\bigg( H^1\left(\Gal{\Q_\Sigma}{\Q_\infty},E[p^2]\right) \rightarrow H^1(I_\eta,\mathfrak{A}[p^2])^{G_\eta/I_\eta} \bigg).
\end{align*}
Note that if $E$ has good ordinary reduction at $p$, then $\mathfrak{A}[p^2]$ is isomorphic to $\overline{E}[p^2]$, the $p^2$-torsion points on the reduced elliptic curve $\overline{E}$ . If $E$ has split multiplicative reduction at $p$, then $\mathfrak{A}[p^2]$ is isomorphic to $\frac{\Z}{p^2\Z}$ with the trivial action of $\Gal{\Q_\Sigma}{\Q}$. The following ``control theorem'' is essentially proved in the work of Greenberg and Vatsal \cite{greenberg2000iwasawa}, relating the $\frac{\Lambda}{(p^2)}$-module $\Sel_{E[p^2]}(\Q_\infty)^\vee$ with the $\frac{\Lambda}{(p^2)}$-module $\Sel_{E[p^\infty]}(\Q_\infty)^\vee \otimes_{\Lambda} \frac{\Lambda}{(p^2)}$:
\begin{proposition}[Proposition 2.8 in \cite{greenberg2000iwasawa}] We have the following isomorphism of $\frac{\Lambda}{(p^2)}$-modules:
\begin{align} \label{control-theorem-elliptic-curve}
\Sel_{E[p^2]}(\Q_\infty)^\vee \cong \Sel_{E[p^\infty]}(\Q_\infty)^\vee \otimes_{\Lambda} \frac{\Lambda}{(p^2)}.
\end{align}
\end{proposition}

\subsection{Proof of Theorem \ref{elliptic-curves}} \mbox{}

We would like to prove the following theorem stated in the introduction:
\ellipticcurves*

\begin{proof}
Recall, from the introduction, that we have the following short exact sequence of modules over $\frac{\Z}{p^2\Z}$ that is $\Gal{\Q_\Sigma}{\Q_\infty}$-equivariant:
{\small \begin{align*}
0 \rightarrow \ker(\Phi) \rightarrow E[p^2] \rightarrow \ker\left(\tilde{\Phi}\right)\rightarrow 0.
\end{align*}}
Consider the long exact sequence in group cohomology, for both $\Gal{\Q_\Sigma}{\Q_\infty}$ and $I_\eta$ (here once again, $\eta$ is the unique prime above $p$ in $\Q_\infty$). We have the following commutative diagram whose rows are exact:
{\footnotesize \begin{align*}
\xymatrix{
0 \ar[r]& H^1 \left(\Gal{\Q_\Sigma}{\Q_\infty}, \ker(\Phi) \right) \ar[d]\ar[r]& H^1 \left(\Gal{\Q_\Sigma}{\Q_\infty}, E[p^2] \right)\ar[d] \ar[r]&H^1 \left(\Gal{\Q_\Sigma}{\Q_\infty}, \ker\left(\tilde{\Phi}\right) \right) \ar[d]\ar[r]& 0. \\
0\ar[r]& 0\ar[r]& H^1\left(I_\eta, \mathfrak{A}[p^2]\right) \ar[r]^{\cong}& H^1\left(I_\eta, \ker\left(\tilde{\Phi}\right)\right) \ar[r]& 0
}
\end{align*}}
To see why the top row is exact, it suffices to show that (i) $H^0 \left(\Gal{\Q_\Sigma}{\Q_\infty},\ker\left(\tilde{\Phi}\right)\right)=0$ and (ii) $H^2 \left(\Gal{\Q_\Sigma}{\Q_\infty},\ker(\Phi) \right)=0$. The first assertion follows from our assumption that the character $\tilde{\phi}$ is odd.  For the second assertion, notice that one obtains an exact sequence,
{\small$$
H^2 \left(\Gal{\Q_\Sigma}{\Q_\infty},\ker(\Phi)[p] \right)  \rightarrow  H^2 \left(\Gal{\Q_\Sigma}{\Q_\infty},\ker(\Phi) \right) \rightarrow H^2 \left(\Gal{\Q_\Sigma}{\Q_\infty},\ker(\Phi)[p] \right),$$} as part of the long exact sequence in group cohomology (for the group $\Gal{\Q_\Sigma}{\Q_\infty}$) using the short exact sequence $0 \rightarrow \ker(\Phi)[p] \rightarrow \ker(\Phi) \rightarrow \ker(\Phi)[p] \rightarrow 0$; the existence of this short exact sequence uses the fact that $\ker(\Phi)$ is cyclic. The arguments given on Page 30 in the work of Greenberg and Vatsal \cite{greenberg2000iwasawa} establish that $H^2 \left(\Gal{\Q_\Sigma}{\Q_\infty},\ker(\Phi)[p] \right)$ equals zero. This forces $H^2 \left(\Gal{\Q_\Sigma}{\Q_\infty},\ker(\Phi) \right)$ to equal zero too. \\

As for the bottom row in the commutative diagram above, it turns out that, by Lemma \ref{kernel-isogenies-ramification}, $\mathfrak{A}[p^2] \cong \ker\left(\tilde{\Phi}\right)$, as modules for the group ring $\frac{\Z}{p^2\Z}[\Gal{\overline{\Q}_p}{\Q_p}]$ (and hence for the action of inertia subgroup $I_\eta$ of $\Gal{\overline{\Q}_p}{\Q_{p,\infty}}$ too). \\

Considering the Snake Lemma for the commutative diagram given earlier, we obtain the short exact sequence $0 \rightarrow \Sel_{\ker(\Phi)}(\Q\infty) \rightarrow \Sel_{E[p^2]}(\Q_\infty) \rightarrow \Sel_{\ker(\tilde{\Phi})} \rightarrow 0$. Taking Pontryagin duals, we obtain the following short exact sequence of $\frac{\Lambda}{(p^2)}$-modules:
\begin{align} \label{ses-elliptic}
0 \rightarrow \Sel_{\ker(\tilde{\Phi})}(Q_\infty)^\vee \rightarrow \Sel_{E[p^2]}(\Q_\infty)^\vee \rightarrow  \Sel_{\ker\left(\ker(\Phi)\right)}(\Q_\infty)^\vee \rightarrow 0.
\end{align}
By (\ref{2-classical-ses-mod-p}) and (\ref{reduction-char-elliptic}), all the $\frac{\Lambda}{(p^2)}$-modules, appearing in (\ref{ses-elliptic}), have projective dimensions less than or equal to one.  Using Lemma 3 in \cite{MR1658000}, we have the following equality of (first) Fitting ideals in $\frac{\Lambda}{(p^2)}$:

\begin{align} \label{equality-of-divisors}
\text{Fitt} \left(\Sel_{E[p^2]}(\Q\infty)^\vee\right) = \text{Fitt}\left(\Sel_{\ker(\Phi)}(\Q_\infty)^\vee\right)  \text{Fitt}\left(\Sel_{\ker(\tilde{\Phi})}(\Q_\infty)^\vee\right).
\end{align}

Using (\ref{2-classical-ses-mod-p}) and (\ref{reduction-char-elliptic}), we have the following equality of (first) Fitting ideals in $\frac{\Lambda}{(p^2)}$:

\begin{align} \label{recall-mod-p2-divisors}
&\text{Fitt} \left( \Sel_{E[p^2]}(\Q_\infty)^\vee\right) = \sigma_{p^2}\bigg(\text{Char}_{\Lambda}\left(\Sel_{E[p^\infty]}(\Q_\infty)^\vee\right)\bigg), \\
\notag &\text{Fitt}\left(\Sel_{\ker(\Phi)}(\Q_\infty)^\vee\right) = \sigma_{\phi}\bigg(\left(\det\left(A_\phi\right)\right)\bigg) , \qquad   \text{Fitt}\left(\Sel_{\ker(\tilde{\Phi})}(\Q_\infty)^\vee\right) = \sigma_{\tilde{\phi}} \bigg(\left(\det\left(A_{\tilde{\phi}}\right)\right)\bigg).
\end{align}

Combining (\ref{equality-of-divisors}) and (\ref{recall-mod-p2-divisors}), we obtain the following equality of ideals in $\frac{\Lambda}{(p^2)}$:
 \begin{align} \label{main-equality-theorem-elliptic}
\sigma_{p^2}\bigg(\text{Char}_{\Lambda}\left(\Sel_{E[p^\infty]}(\Q_\infty)^\vee\right)\bigg) = \sigma_{\phi}\bigg(\left(\det\left(A_\phi\right)\right)\bigg)  \sigma_{\tilde{\phi}} \bigg(\left(\det\left(A_{\tilde{\phi}}\right)\right)\bigg).
\end{align}
This completes the proof of Theorem \ref{elliptic-curves}.

\end{proof}

\section*{Acknowledgements}
Both authors thank Ralph Greenberg for his guidance in graduate school, and for his steadfast encouragement to write this paper. The second author thanks Ted Chinburg for indicating to him the works of Reiner and Ullom. We thank Andreas Nickel for many helpful comments. We are also grateful to the anonymous referees for their feedback and providing many thoughtful suggestions.

\bibliographystyle{abbrv}
\bibliography{biblio}
\end{document}